\documentclass[12pt]{amsart}

\usepackage[utf8]{inputenc}
\usepackage[USenglish]{babel}
\usepackage{amsmath,amsthm,amssymb,amsfonts}
\usepackage{mathtools}
\usepackage{mathrsfs}
\usepackage{bbm}

\usepackage{indentfirst}
\usepackage{microtype}
\usepackage{enumerate}
\usepackage{accents}
\usepackage{cancel} 
\usepackage{xifthen}

\usepackage{multirow}
\usepackage{longtable}
\usepackage{enumitem}

\usepackage{float}
\usepackage{graphicx}
\usepackage{placeins}
\usepackage{caption}
\usepackage{arydshln}

\usepackage{tikz}
\usetikzlibrary{patterns}
\usepackage{tikz-cd}

\usetikzlibrary{decorations.pathreplacing,decorations.markings}
\usetikzlibrary{babel}

\tikzset{                        
    symbol/.style={
        ,draw=none
        ,every to/.append style={%
            edge node={node [sloped, allow upside down, auto=false]{$#1$}}}
    }
}

\usepackage{chemgreek,textgreek} 

\usepackage[breaklinks=true, bookmarksopenlevel=1, bookmarksdepth=2]{hyperref}

\allowdisplaybreaks

\newtheorem{thm}{Theorem}[section]
\newtheorem{cor}[thm]{Corollary}
\newtheorem{lem}[thm]{Lemma}
\newtheorem{prop}[thm]{Proposition}
\newtheorem{conj}[thm]{Conjecture}

\theoremstyle{plain} 
\newcommand{\thistheoremname}{}
\newtheorem*{genericthm}{\thistheoremname}
\newenvironment{nthm}[1]
  {\renewcommand{\thistheoremname}{#1}%
   \begin{genericthm}}
  {\end{genericthm}}

\theoremstyle{definition}

\theoremstyle{remark}
\newtheorem{rem}[thm]{Remark}
\newtheorem*{xrem}{Remark}
\newtheorem*{ntt}{Notation}

\numberwithin{equation}{section}

\newcommand{\Z}{\mathbb{Z}}      
\newcommand{\Q}{\mathbb{Q}}      
\newcommand{\R}{\mathbb{R}}      
\newcommand{\C}{\mathbb{C}}      
\newcommand{\uhp}{\mathfrak{h}}  

\newcommand{\eps}{\varepsilon}   

\newcommand{\norm}[1]{\left\lVert #1\right\rVert} 

\newcommand\restr[2]{{           
  \left.\kern-\nulldelimiterspace #1%
  \right|_{#2}%
 }}

\newcommand\minus{               
  \setbox0=\hbox{-}%
  \vcenter{%
    \hrule width\wd0 height \the\fontdimen8\textfont3%
  }%
}

\newcommand{\SL}{\mathrm{SL}}    

\renewcommand{\pmod}[1]{         
  ~(\mathrm{mod}~#1)}

\newcommand{\lmod}[3]{         
  #1 \equiv #2 \pmod{#3}}

\renewcommand{\O}{\mathcal{O}}                          

\tikzset{                        
    symbol/.style={%
        ,draw=none
        ,every to/.append style={%
            edge node={node [sloped, allow upside down, auto=false]{$#1$}}}
    }
}

\newcommand{\negphantom}[1]{\settowidth{\dimen0}{$\displaystyle #1$}\hspace*{-\dimen0}}

\newcommand{\hei}{\mathrm{ht}}
\newcommand{\lht}[1]{\mathrm{ht}([ #1])}
\newcommand{\fhr}[2]{\mathcal{N}_{#2}([ #1])}
\newcommand{\CCC}[1]{\mathscr{C}^{\times}\ifthenelse{\isempty{#1}}{}{(#1)}}
\newcommand{\ccc}[1]{\mathcal{C}\ifthenelse{\isempty{#1}}{}{(#1)}}
\newcommand{\FD}{\mathscr{F}}
\newcommand{\rtd}{\mathrm{rd}}
\newcommand{\hilb}[1]{\mathrm{H}_{#1}}
\newcommand{\hilbt}[1]{\widetilde{\mathrm{H}}_{#1}}
\newcommand{\kron}{\mathfrak{K}}
\newcommand{\sumq}[1]{\ \sideset{}{^{(#1)}}\sum_{(a,b,c)}}

\newcommand{\Cl}{\mathrm{C}\ell}

\newcommand{\Li}{\mathrm{Li}}

\newcommand{\A}{\mathscr{A}}

\restylefloat{table}

\setlist[itemize]{leftmargin=*}
\setlist[enumerate]{leftmargin=*}

\makeatletter
 \newcommand*\alphgreek[1]{\expandafter\@alphgreek\csname c@#1\endcsname}
 \newcommand*\@alphgreek[1]{\csname chemgreek_int_to_greek:n\endcsname{#1}}
 \newcommand*\Alphgreek[1]{\expandafter\@Alphgreek\csname c@#1\endcsname}
 \newcommand*\@Alphgreek[1]{\csname chemgreek_int_to_Greek:n\endcsname{#1}}

 \AddEnumerateCounter*{\alphgreek}{\@alphgreek}{\chemalpha}
 \AddEnumerateCounter*{\Alphgreek}{\@Alphgreek}{\chemAlpha}
\makeatother

\makeatletter
\def\@tocline#1#2#3#4#5#6#7{\relax
  \ifnum #1>\c@tocdepth 
  \else
    \par \addpenalty\@secpenalty\addvspace{#2}%
    \begingroup \hyphenpenalty\@M
    \@ifempty{#4}{%
      \@tempdima\csname r@tocindent\number#1\endcsname\relax
    }{%
      \@tempdima#4\relax
    }%
    \parindent\z@ \leftskip#3\relax \advance\leftskip\@tempdima\relax
    \rightskip\@pnumwidth plus4em \parfillskip-\@pnumwidth
    #5\leavevmode\hskip-\@tempdima
      \ifcase #1
       \or\or \hskip 1em \or \hskip 2em \else \hskip 3em \fi%
      #6\nobreak\relax
    \dotfill\hbox to\@pnumwidth{\@tocpagenum{#7}}\par
    \nobreak
    \endgroup
  \fi}
\makeatother

\begin{document}

\title{On Landau--Siegel zeros and heights of singular moduli}%
\author{Christian T\'afula}%
\address{D\'epartment de Math\'ematiques et Statistique, %
 Universit\'e de Montr\'eal, %
 CP 6128 succ Centre-Ville, %
 Montreal, QC H3C 3J7, Canada}%
\email{christian.tafula.santos@umontreal.ca}%


\begin{abstract}
 Let $\chi_D$ be the Dirichlet character associated to $\mathbb{Q}(\sqrt{D})$ where $D < 0$ is a fundamental discriminant.
 Improving Granville--Stark \cite{grasta00}, we show that
 \[ \frac{L'}{L}(1,\chi_D) = \frac{1}{6}\, \mathrm{height}(j(\tau_D)) - \frac{1}{2}\log|D| + C + o_{D\to -\infty}(1), \]
 where  $\tau_D =   \frac 12(-\delta+\sqrt{D})$ for $D \equiv \delta ~(\mathrm{mod}~4)$ and
 $j(\cdot)$ is the $j$-invariant function with $C = -1.057770\ldots$. Assuming the ``uniform'' $abc$-conjecture for number fields, we deduce that   $L(\beta,\chi_D)\ne 0$ with 
$\beta \geq 1 - \frac{\sqrt{5}\varphi + o(1)}{\log|D|}$ 
 where $\varphi = (1+\sqrt{5})/2$, which we improve for smooth $D$.
\end{abstract}
\maketitle

\section{Introduction}
 In 2000, Granville and Stark \cite{grasta00} showed that the uniform $abc$-conjecture for number fields (Conjecture \ref{polabc} (iii)) implies that there are no ``Siegel zeros'' for odd characters, by deducing that, under uniform $abc$, the class number of $\Q(\sqrt{D})$ satisfies  
 \begin{equation}
  h(D) \geq (1+o(1))\, \frac{\pi}{3}\frac{\sqrt{|D|}}{\log|D|} \sumq{D}\frac{1}{a} \label{grst2}
 \end{equation}
 where the sum runs over the reduced binary quadratic forms $Q(x,y) = ax^2+bxy+cy^2$ of fundamental discriminant $D<0$. Comparing this lower bound to the unconditional, more traditional type of lower bound: 
 \begin{equation}
  h(D) = \left(\frac{1 + O\big(\hspace{-.1em}\log\log|D|/\log|D|\big)}{1+ 2\,\frac{L'}{L}(1,\chi_{D})/\log|D|}\right) \frac{\pi}{3} \frac{\sqrt{|D|}}{\log|D|} \sumq{D}\frac{1}{a}, \label{grst1}
 \end{equation}
 where $\chi_D$ is the corresponding quadratic character, one derives from uniform $abc$ that $|\frac{L'}{L}(1,\chi_D) | \ll \log|D|$, which is equivalent to $L(s,\chi_D)$ having no ``Siegel zeros''. Our main theorems imply a slightly more precise estimate than \eqref{grst1}, and the existence of a subsequence of $D$'s for which $\frac{L'}{L}(1,\chi_D) = o_{D\to -\infty}(\log |D|)$.

 For $D<0$ we need the generator of the ring of integers of $\Q(\sqrt{D})$,
 \[ \tau_D = \begin{cases} \frac 12(-1+\sqrt{D}) & \text{ if } \lmod{D}{1}{4},\\
 \frac 12 \sqrt{D} & \text{ if } \lmod{D}{0}{4}; \end{cases} \]
 and let $j(\cdot)$ be  the classical $j$-invariant function.
 Define the \emph{height} of $m/n\in \mathbb Q$ with $(m,n)=1$ by 
 \[ \hei(m/n) = \log\max\{|m|,|n|\}. \]
 We will give the \emph{height function} for algebraic numbers in subsection \ref{ssec16}. We improve on \eqref{grst1} by using Duke's equidistribution theorem \cite{duk88} (Lemma \ref{dkps}): 
  
 \begin{thm}\label{thm02}
  For fundamental discriminants $D<0$, we have
  \begin{equation}\label{L'/L}
   \frac{L'}{L}(1,\chi_D) = \frac{1}{6}\, \hei(j(\tau_D)) - \frac{1}{2}\log|D| + C + o_{D\to -\infty}(1),
  \end{equation}
  where $C := \kappa_1 - \kappa_2 + \kappa_3 + \gamma - \log 2 = -1.057770\ldots$, and $\gamma = 0.577215\ldots$ is Euler--Mascheroni's constant.\footnote{See Lemmas \ref{explCONST1}, \ref{explCONST2}, \ref{explCONST3} for the definition of $\kappa_1$, $\kappa_2$, $\kappa_3$, respectively.} In particular, this implies that
  \begin{equation}
   h(D) = \left(\frac{1 + O\big(1/\log|D|\big)}{1+ 2\,\frac{L'}{L}(1,\chi_{D})/\log|D|}\right) \frac{\pi}{3} \frac{\sqrt{|D|}}{\log|D|} \sumq{D}\frac{1}{a} . \label{hdbtt}
  \end{equation}
 \end{thm}

 Expressing $\frac{L'}{L}(1,\chi_D)$ in terms of $\hei(j(\tau_D))$ allows us to understand $L$-function values using  Diophantine geometry techniques: The classical \emph{$abc$-conjecture} (the \emph{Masser--Oesterl\'e conjecture}) states that for coprime integers $a,b,c $  satisfying $a+b=c$, we have
 \[ \log\max\big\{|a|,|b|,|c|\big\} \leq (1+\eps) \sum_{p\,\mid\, abc} \log p + O_\eps(1), \]
 where the implied constant depends only on $\eps>0$. Several important results in number theory  follow from this statement, including ``asymptotic  Fermat'' (see Example 5.5.2, p. 71--72 of Vojta \cite{vojta87}). Granville and Stark used  a uniform extension of the $abc$-conjecture to number fields.\footnote{Formulated so as to be consistent between different field extensions.} The tricky part is deciding what to conjecture about the contribution of ramified primes in this inequality:  thus, if $\rtd_K := |\Delta_K|^{1/[K:\Q]}$ then we study two versions, where one includes the error term  
 \begin{center}
 $O(\log(\rtd_K))$ (\emph{$O$-weak uniformity}) or $o(\log(\rtd_K))$ (\emph{weak uniformity}),
 \end{center}
 given by Conjectures \ref{polabc} (i) and (ii), respectively.
 With minor modifications to the argument, we prove the following version of Granville--Stark's theorem:
 
 \begin{thm}\label{thm03}  
  Assuming the $abc$-conjecture for number fields with 
  \begin{itemize} 
   \item  $O$-weak uniformity we have $|\frac{L'}{L}(1,\chi_D)| \ll \log|D|$. \smallskip
   \item  weak uniformity we have $\limsup\limits_{D\to-\infty} \frac{1}{\log|D|}\frac{L'}{L}(1,\chi_D) = 0$. 
  \end{itemize}
 \end{thm}

\subsection{Traditional approach to \texorpdfstring{$\frac{L'}{L}(1,\chi_D)$}{L'/L(1,chi\_D}}
 Let $\varrho(\chi)$ be the  set  of non-trivial zeros of $L(s,\chi)$, each included in the set with multiplicity. We are interested in the zeros of $L(s,\chi)$ inside the region
 \begin{equation}
  \mathcal{R}_A = \mathcal{R}_{A}(q) := \bigg\{s\in\C ~\bigg|~ \sigma \geq 1 - \frac{A}{\log q},\ |t| \leq 1 \bigg\}. \label{Rzfr}
 \end{equation}
 for any given $A\in\R_{\geq 1}$ with $q>e^A$. A \emph{Siegel zero} is a real zero $\beta$ of $L(s,\chi)$ for some primitive quadratic $\chi \pmod{q}$, with
 $\beta\geq 1 - \frac{A}{\log q}$ for some given $A>0$.
  
 \begin{thm}\label{thm01} For every primitive character $\chi \pmod q$ we
  have the bound $|\varrho(\chi)\cap \mathcal{R}_A| \ll e^{3A}$ (with an absolute implicit constant), and
  \[ \sum_{\varrho(\chi)\cap\mathcal{R}_A} \Re\bigg(\frac{1}{1-\varrho}\bigg) < \bigg(1- \frac{1}{\sqrt{5}}\bigg) \frac{1}{2} \log q + \Re\bigg(\frac{L'}{L}(1,\chi)\bigg) + O(e^{3A}). \]       
 \end{thm}
 
 An integer $q$ is $y$-\emph{smooth} if all its prime factors are $\leq y$. Chang \cite{cha14} established wide zero-free regions for such $q$ (see subsection \ref{ssecCZFR}) so we deduce:
 
 \begin{thm}\label{thm01B}
  For every $\delta>0$ there is an $N_{\delta}>0$ such that if  $q>N_{\delta}$  is  $q^{\delta}$-smooth then the only possible element of $\varrho(\chi)\cap \mathcal{R}_{A}$ with $A\ll 1/\delta$, is the potential Siegel zero $\beta = \beta_{\chi} := \max\{\sigma \in \R ~|~ L(\sigma,\chi) = 0\}$.  In that case
  \[ \Re\bigg(\frac{L'}{L}(1,\chi)\bigg) = \frac{\mathbbm{1}_{\chi=\overline{\chi}}}{1-\beta} + O(\delta^{\frac{1}{2}}\log q). \]
  \end{thm}
 
 Theorem \ref{thm01B} implies that, for every fixed $\delta>0$, if $D<0$ is a $|D|^{\delta}$-smooth fundamental discriminant then
 \begin{equation}
  \frac{L'}{L}(1,\chi_D) \geq -M\delta^{\frac{1}{2}}\log|D|  \label{liminfd}
 \end{equation}
 for some $M\in\R_{\geq 0}$ not depending on $\delta$, provided $|D|$ is sufficiently large.

\subsection{Main corollaries}
 We present two results that are consequences of the theorems above. First, by combining Theorems \ref{thm03}, \ref{thm01} and \ref{thm01B} together, we deduce the following:
   
 \begin{cor}\label{MC1}
  Assume the weak uniform $abc$-conjecture. As $D\to -\infty$, the function $L(s,\chi_D)$ has no zeros in the real interval
  \[ \bigg[1 - \dfrac{\sqrt{5}\varphi + o(1)}{\log |D|},\ 1\bigg] \]
  where $\varphi := \frac{1+\sqrt{5}}{2}$, nor in the region
  \[ \bigg\{s\in\C ~\bigg|~ \sigma \geq 1 - \dfrac{m\delta^{-\frac{1}{2}}+o_{\delta}(1) }{\log |D|},\ |t|\leq 1\bigg\} \]
  (for some absolute constant $m>0$) when $D$ is $|D|^{\delta}$-smooth for given $\delta>0$.
 \end{cor}
 
 See subsection \ref{ssecMC1} for a proof. Next we give explicit upper and lower bounds on several key analytic quantities:
 
 \begin{cor}\label{MC2}
  For negative fundamental discriminants $D$ we have:
  \begin{enumerate}[label=(\roman*)]
   \item $\displaystyle \frac{3}{\sqrt{5}} \log|D|+O(1) \,\leq\, \hei(j(\tau_D)) \,\stackrel{*}{\leq}\, (1+o(1))\, 3\log|D|$.
   \item $\displaystyle \frac{1}{2\sqrt{5} }\log|D|+O(1) \ \leq\ \sum_{\varrho(\chi_D)}\frac{1}{\varrho} \ \stackrel{*}{\leq}\ (1 + o(1))\, \frac{1}{2}\log|D|$.
   \item $\displaystyle (1 + o(1))\, \frac{\pi}{3}\frac{\sqrt{|D|}}{\log|D|} \sum_{(a,b,c)}\frac{1}{a} \,\stackrel{*}{\leq}\, h(D) \,\leq\, \text{\small$\Bigg(\sqrt{5} +O\bigg(\frac{1}{\log|D|}\bigg)\Bigg)$}\, \frac{\pi}{3}\frac{\sqrt{|D|}}{\log|D|} \sum_{(a,b,c)}\frac{1}{a}\hspace{-30pt}$
  \end{enumerate}
  where the starred inequalities ``$\stackrel{*}{\leq}$''  are conditional on the weak uniform $abc$-conjecture.
  Moreover, if $D$ is $|D|^{o(1)}$-smooth then these starred inequalities become (asymptotic) starred equalities.
  If we assume the Generalized Riemann Hypothesis then we obtain these asymptotic equalities for all $D$, with error term a factor of 
  $O(\frac{\log\log|D|}{\log|D|} )$.\footnote{For example, $\hei(j(\tau_D))  = 3 \log |D|+O( \log\log|D| )$ assuming the Generalized Riemann Hypothesis.}
 \end{cor}
 
 This corollary follows from Theorems \ref{thm02} and \ref{thm03}; part (ii) is obtained from \eqref{pi01}, and part (iii) from \eqref{hdbtt}. (We omit the details.)
 In each case the upper and lower bounds in Corollary \ref{MC2} differ by an artificial  multiplicative factor that is $\sim \sqrt{5}$ arising from our proof 
 (see Lemma \ref{cuteineq} (ii)), which may be improvable. Data supports part (i) of Corollary \ref{MC2}:
 
 \FloatBarrier
 \begin{figure}[hbt!]
  \centering
  \captionsetup{type=figure}
  \includegraphics[width=.9\textwidth]{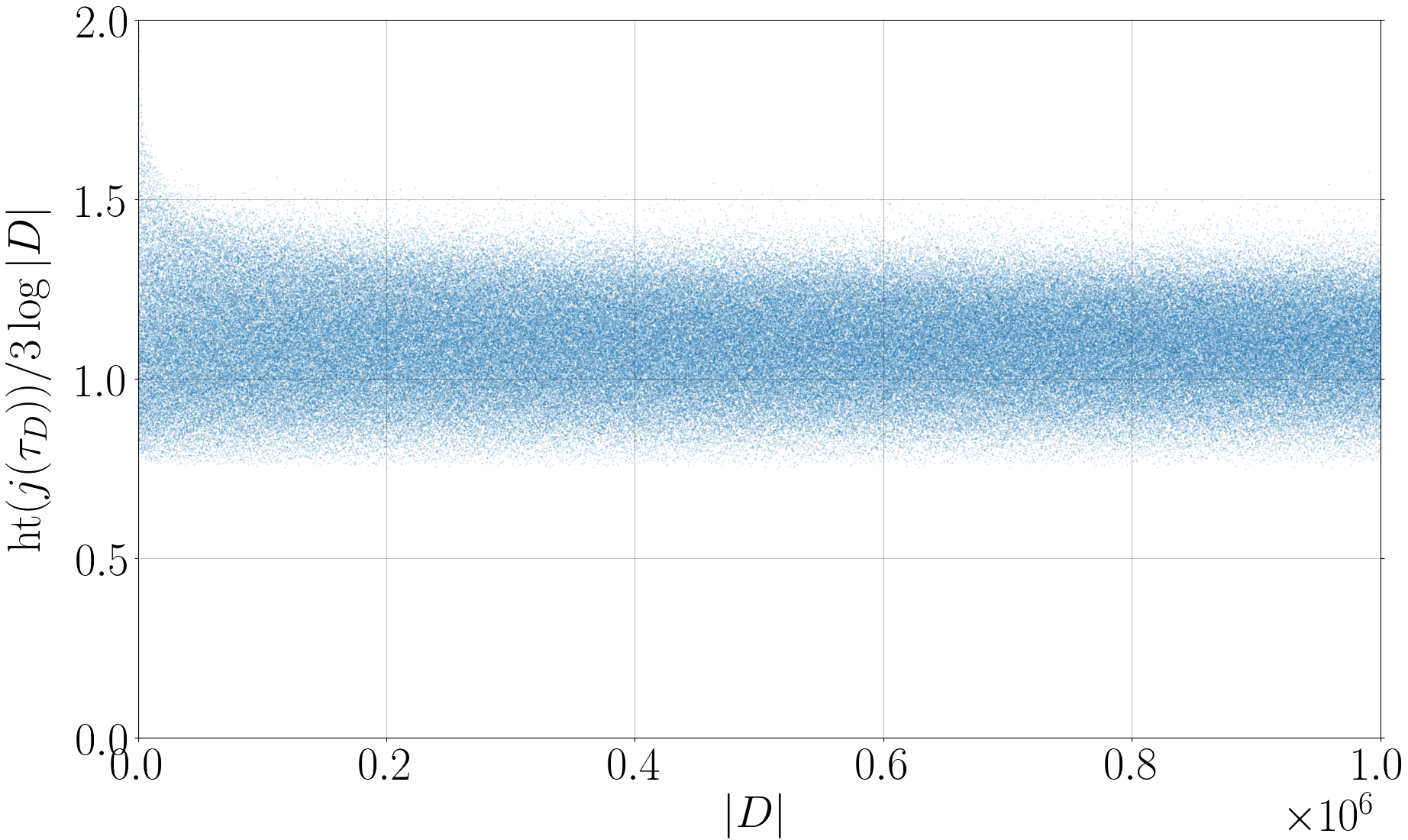}
  \captionof{figure}{Graph of $\frac{\hei(j(\tau_D))}{3\log|D|}$ for $ -10^6\leq D\leq 0$.}
  \label{fig1}
 \end{figure}
 
 The perceptible slight bias towards $>1$ in Figure \ref{fig1} can be analyzed by rewriting \eqref{L'/L} as 
 \begin{equation*}
  \frac{\hei(j(\tau_D))}{3\log|D|} = 1 + \frac{2( \frac{L'}{L}(1,\chi_D) -C+o(1))}{\log |D|}.
 \end{equation*}
 Therefore this bias towards $>1$ in Figure \ref{fig1} must stem from  $\frac{L'}{L}(1,\chi_D) - C$ being usually positive. We verify this observation in Figure \ref{fig2}, by noting that $\frac{L'}{L}(1,\chi_D)$ is usually larger than $C = -1.057770\ldots$.
  
 \begin{figure}[hbt!]
  \centering
  \captionsetup{type=figure}
  \includegraphics[width=1\textwidth]{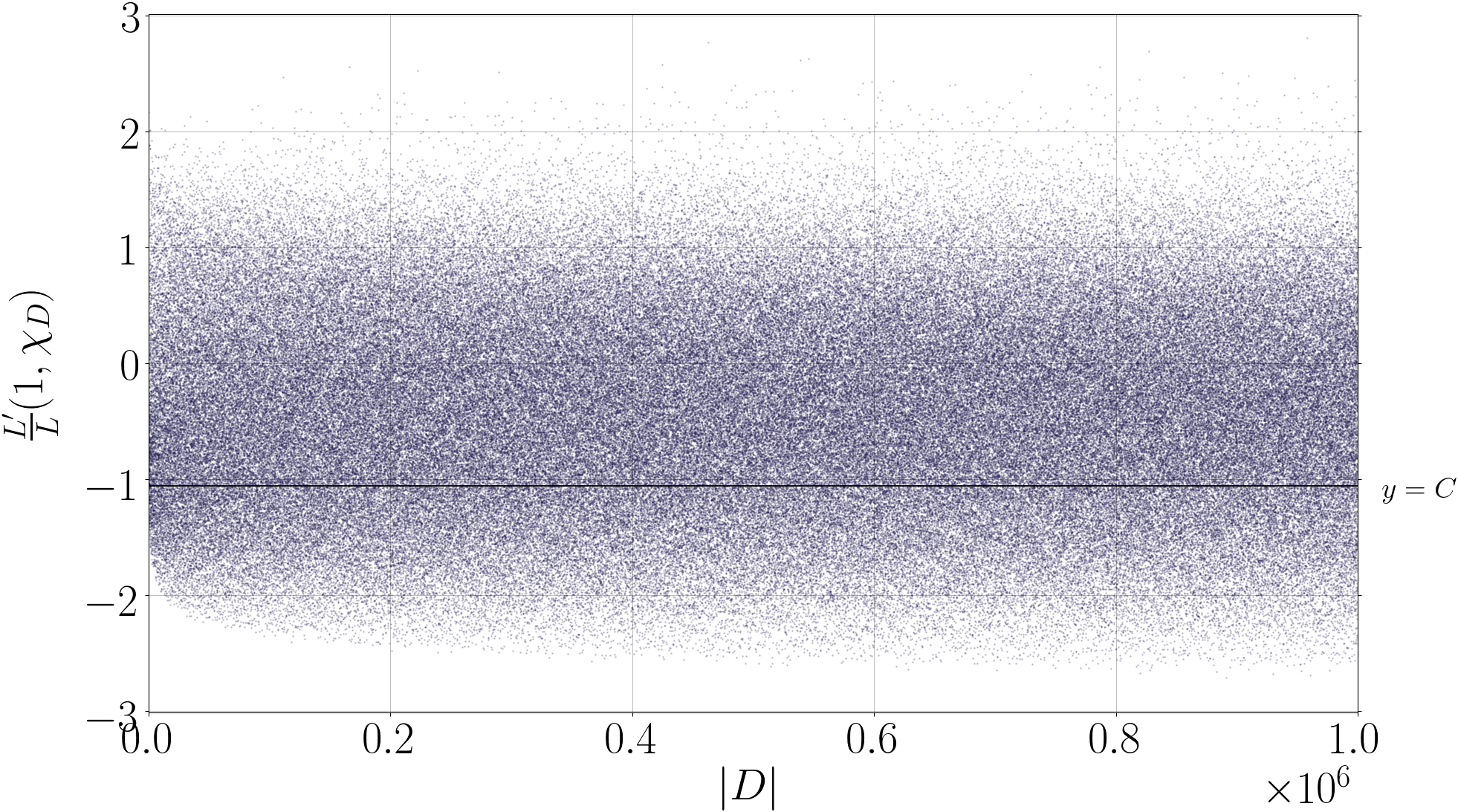}
  \captionof{figure}{Graph of $\frac{L'}{L}(1,\chi_D)$, for $-10^6 \leq D < 0$.}
  \label{fig2}
 \end{figure}
 
 Lastly, an appendix on the Hadamard factorization of $L$-functions following Lucia \cite{luciaHDM} is included, justifying a formula we state in section \ref{sec2}.
 
 \begin{ntt}
  Let $f,g: [x_0,+\infty) \to \R_{+}$ be positive real-valued functions defined in $[x_0,+\infty)$ for some $x_0\in\R$. We write $f\sim g$, $f=o(g)$, or $f = O(g)$, if $|f(x)/g(x)|$ goes to $1$, $0$, or stays bounded as $x\to +\infty$, respectively. If $f = O(g)$, we also write $f \ll g$. If $f\ll g$ and $g\ll f$, then we write $f\asymp g$.
 \end{ntt}

\section{Definitions and notation}\label{ssecdefs}
 We review some notation and results that will be used in this paper:
 
\subsection{Heegner points}\label{ss12}
 We denote the binary, primitive quadratic form
 \[
 Q(x,y) = ax^2 + bxy + cy^2 \in \Z[x,y]
 \]
  by $Q = (a,b,c)$ which has discriminant $d = b^2 - 4ac$. If $d<0$ then we will assume that $a>0$ so that $Q$ is positive-definite.
  Two forms $Q_1$, $Q_2$ are \emph{equivalent} if 
  \[
  Q_1(x,y) = Q_2(\alpha x + \beta y, \gamma x + \delta y) \text{  for some } \begin{psmallmatrix} \alpha & \beta \\ \gamma & \delta \end{psmallmatrix} \in \SL_2(\Z).
  \]
    For every non-zero discriminant $d\equiv 0$ or $1 \pmod{4}$, we have the principal form
 \begin{equation}
  Q_{1,d}(x,y) := \begin{cases}
    x^2 - \frac{d}{4}y^2\phantom{x^2 + xy + \frac{1-d}{4}y^2}\negphantom{\text{$x^2 - \frac{d}{4}y^2$}} \ \,\text{ if } d \equiv 0 \pmod{4}, \\
    x^2 + xy + \frac{1-d}{4}y^2 \ \text{ if } d \equiv 1 \pmod{4}.
   \end{cases} \label{grundform}
 \end{equation}
 For $d<0$, we say that $(a,b,c)$ is \emph{reduced} if $-a<b\leq a<c$ or $0\leq b\leq a = c$; every quadratic form of discriminant $d$ is equivalent to a single reduced form. If $Q$ is reduced, then $a\leq \sqrt{|d|/3}$.
 
 \noindent
 \edef\myindent{\the\parindent} 
 \begin{flushleft}
 \begin{minipage}{0.61\textwidth}\setlength{\parindent}{\myindent}
  {
 Each binary quadratic form $(a,b,c)$ with $d<0$   is associated with  
 a \emph{CM point} 
 \[
 \tau_{(a,b,c)} = \frac{-b+\sqrt{d}}{2a} \in \uhp  := \{z \in \C ~|~ \Im(z)>0\} .
 \]
 $(a,b,c)$ is reduced if and only if $\tau_{(a,b,c)}$ lies in the  fundamental domain, $\FD$, depicted on the right.   \emph{Heegner points} are the CM points associated to the reduced forms of \emph{fundamental} discriminants $D<0$; i.e., $D$'s which are the discriminant of some 
   \unskip\parfillskip 0pt \par
  }\vskip.0\baselineskip
 \end{minipage}\, 
 \begin{minipage}{0.37\textwidth}
 \begin{flushright}
   \vspace{-1em}
   
   \begin{tikzpicture}[scale=1.2, every node/.style={scale=1.2}]
    \draw[help lines, ystep=1, xstep=1, black!20] (-2.1,-.1) grid (2.1,4.1);
    \draw[->, black] (0,-.1)--(0,4.1);
    \draw[->, black] (-2.2,0)--(2.2,0);
    
    \foreach \y in {-1, -.5, 0, .5, 1}
    \draw (2*\y,0)--(2*\y,-0.1);
    \foreach \x in {-1, -0.5, 0, 0.5, 1}
    \draw (2*\x,-0.3) node[anchor=south west, scale=.5] {\x};
    
    \fill[pattern=north east lines, pattern color=black!50] (-1,4) -- (1,4) -- (1,1.732) to[out=150,in=0] (0,2) to[out=180,in=30] (-1,1.732) -- (-1,4);    
    \draw[-, dotted, black] (-2,0) to[out=90,in=180] (0,2) to[out=0,in=90] (2,0);
    \draw[-, black] (-1,1.732)--(-1,4);
    \draw[-, dashed, black] (1,1.732)--(1,4);
    \draw[-, dashed, black] (1,1.732) to[out=151.5,in=-1.5] (0,2);
    \draw[-, black] (0,2) to[out=181.5,in=28.5] (-1,1.732);
    
    \draw (-1,1.732) node[scale=0.7] {\footnotesize{$\bullet$}};
    \draw (1,1.732) node[scale=0.7] {$\circ$};
    \draw (-1.5, 1.732 +.25) node[preaction={fill, white}, scale=0.9] {\footnotesize{$e^{\frac{2\pi i}{3}}$}};
    \draw (1.5, 1.732 +.25) node[preaction={fill, white}, scale=0.9] {\footnotesize{$e^{\frac{\pi i}{3}}$}};
    
    \draw (.5, 3.5) node[preaction={fill, white}, scale=1] {\footnotesize{$\FD$}};
   \end{tikzpicture}
   
   \vspace{-1em}
 \end{flushright}
 \vskip.0\baselineskip
 \end{minipage}
 \end{flushleft}
 quadratic number field. Write $\Lambda_D$ for the set of Heegner points of $D$, and $\Cl(D)$ for the ideal class group of $\Q(\sqrt{D})$. By the classical correspondence
 \begin{equation}
  \Cl(D) \ni \A \ \ \leftrightarrow\ \ \begin{array}{@{}c@{}} \text{reduced } (a,b,c) \\[.2em] (b^2 - 4ac = D) \end{array} \ \ \leftrightarrow\ \ \tau = \dfrac{-b+\sqrt{D}}{2a} \in \Lambda_D \ (\subseteq \FD), \label{corresp}
 \end{equation}
 we may index Heegner points by $\Cl(D)$, namely $\Lambda_D = \{\tau_{\A} ~|~ \A\in\Cl(D)\}$. The Heegner point associated to the trivial ideal class in $\Cl(D)$, which corresponds to the principal form \eqref{grundform}, is denoted $\tau_D$.

\subsection{Singular moduli}\label{ss15}
 For $\tau\in\uhp$, write $q = q_\tau := e^{2\pi i \tau}$. The \emph{$q$-expansion} of the \emph{$j$-invariant function} has the form
 \begin{equation}
  j(\tau) := \frac{\Big(1+ 240\, \sum_{n\geq 1} \big(\sum_{d\mid n} d^3\big) q^n\Big)^3}{q\, \prod_{n\geq 1} (1-q^n)^{24}} = \frac{1}{q} + \sum_{n\geq 0} c(n) q^n, \label{qexpjinv}
 \end{equation}
 with $c(0) = 744$, $c(1) = 196884$, and $\Z_{\geq 1} \ni c(n) \sim \frac{1}{\sqrt{2}} e^{4\pi\sqrt{n}}n^{-3/4}$. This is the unique modular function with respect to $\SL_2(\Z)$ of weight $0$, holomorphic in $\uhp$, satisfying $j(e^{2\pi i/3})= 0$, $j(i) = 1728$, and having a simple pole at $i\infty$. The values taken by $j(\tau)$ at CM points $\tau\in \uhp$ are called \emph{singular moduli}.
 
 For $D<0$ a fundamental discriminant, write $\hilb{D}$ for the \emph{Hilbert class field} of $\Q(\sqrt{D})$, which is the maximal unramified abelian extension of $\Q(\sqrt{D})$. Then:
 \begin{itemize}
  \item $\hilb{D} = \Q(\sqrt{D}, j(\tau_D))$ (and $[\hilb{D}: \Q(\sqrt{D})] = [\Q(j(\tau_D)) : \Q] = h(D)$);\smallskip
  
  \item $\{ j(\tau) ~|~ \tau \in \Lambda_D \}$ is the complete set of $\mathrm{Gal}(\overline{\Q}/\Q(\sqrt{D}))$-conjugates of $j(\tau_D)$;\smallskip
  
  \item $j(\tau_D)$ is an algebraic integer.
 \end{itemize}
 
\subsection{Heights and conductors}\label{ssec16}
 For a number field $K/\Q$ and a place $v$ on $K$, write $K_v$ for the completion of $K$ with respect to $v$. For non-archimedean $v$, write $\mathfrak{p}_{v}$ for the prime ideal in $\mathcal{O}_K$ associated to $v$, $p_v$ for the positive rational prime below $\mathfrak{p}_v$, and $f_v$ for the degree of the residual extension $K_v/\Q_{p_v}$. If $v$ is archimedean, let $\iota_v: K\hookrightarrow \C$ denote an embedding which induces $v$. For $x\in K^{\times}$, define the \emph{normalized absolute value} $\norm{\cdot}_{v}$ as:
 \[ \norm{x}_{v} := \begin{cases}
                     |\iota_v(x)|\negphantom{|\iota_v(x)|}\phantom{p_v^{-f_v \,\mathrm{ord}_{\mathfrak{p}_v}(x)}} \quad \text{ if } K_v \simeq \R, \\
                     |\iota_v(x)|^2\negphantom{|\iota_v(x)|^2}\phantom{p_v^{-f_v \,\mathrm{ord}_{\mathfrak{p}_v}(x)}} \quad \text{ if } K_v \simeq \C, \\
                     p_v^{-f_v \,\mathrm{ord}_{\mathfrak{p}_v}(x)} \quad \text{ if } v\text{ is non-archimedean},
                    \end{cases}
 \]
 where $\mathrm{ord}_{\mathfrak{p}_v}(x)$ ($=: v(x)$) denotes the power of $\mathfrak{p}_v$ appearing in the prime factorization of the principal fractional ideal $(x)\subseteq K$, and $|\cdot|$ is the usual absolute value in $\C$.
 
 Let $\mathcal{M}_K$ denote the set of inequivalent places of $K$ satisfying the \emph{product formula} $\prod_{v\in\mathcal{M}_K} \norm{x}_{v} = 1$ for all $x\in K^{\times}$, and write $\mathcal{M}^{\text{non}}_K \subseteq \mathcal{M}_K$ for the subset of the non-archimedean places. For a point $P = [x_0:\ldots:x_n] \in \mathbb{P}_{K}^{n}$ in the projective $n$-space over $K$, we define its \emph{(na\"{i}ve, absolute, logarithmic) height} by
 \[ \hei(P) := \frac{1}{[K:\Q]} \sum_{v\in\mathcal{M}_{K}} \log \max_{i\leq n}\{\norm{x_i}_v\}, \]
 and its \emph{(logarithmic) conductor} by
 \[ \mathcal{N}_K(P) := \frac{1}{[K:\Q]} \sum_{\substack{v\in\mathcal{M}^{\mathrm{non}}_{K} \\ \exists i,j \leq n ~\mathrm{s.t.}\\ v(x_i)\neq v(x_j)}} f_v \log(p_v). \]
 Neither $\hei$ nor $\mathcal{N}$ depend on the choice of representatives for $P \in \mathbb{P}_{K}^{n}$, since, for any $c\in K^{\times}$, we have $\hei(P) = \hei(cP)$ from the product formula, and $\mathcal{N}_K(P) = \mathcal{N}_K(cP)$ since the sum runs over the same places. Moreover, the height does not depend on the choice of the base field (provided its $\mathbb{P}^{n}$ contains $P$); however, the conductor does.

 Finally, for $x \in K^{\times}$, write $\hei(x) := \lht{x:1}$. If $\alpha\in\overline{\Q}^{\times}$ is integral, then
 \begin{equation}
  \hei(\alpha) = \frac{1}{|\mathcal{A}|}\sum_{\alpha^{*} \in \mathcal{A}} \log^{+}|\alpha^{*}|, \label{heiint}
 \end{equation}
 where $\log^{+}x := \log\max\{1, x\}$ for $x\in \R_{>0}$, $\mathcal{A}$ is the complete set of $\mathrm{Gal}(\overline{\Q}/\Q)$-conjugates of $\alpha$, and the absolute value is being taken on some fixed embedding $\iota: \overline{\Q} \hookrightarrow \C$ (to which the value of $\hei(\alpha)$ is independent).

\section{On \texorpdfstring{$\Re(\frac{L'}{L}(1,\chi))$}{Re(L'/L(1,chi))} and zeros near \texorpdfstring{$s=1$}{s=1}}\label{sec2}
 Consider primitive characters $\chi\pmod{q}$ for $q\geq 2$, and write $\varrho(\chi)$ for the set of non-trivial zeros of $L(s,\chi)$ counted with multiplicity. In this section we are going to prove the following two propositions:
 
 \begin{prop}\label{AnLem}
  Let $\mathcal{S}$ be any finite subset of non-trivial zeros of $L(s,\chi)$ (including the empty set). Then:
  \[ \sum_{\varrho \in \mathcal{S}} \Re\bigg(\frac{1}{1-\varrho}\bigg) < \bigg(1- \frac{1}{\sqrt{5}}\bigg) \frac{1}{2} \log q + \Re\left(\frac{L'}{L}(1,\chi)\right) + \bigg(1+ \frac{1}{\sqrt{5}}\bigg)\frac{2|\mathcal{S}|+3}{2} - 1. \]
 \end{prop}
  
 \begin{prop}\label{AnLem2}
  Consider the region
  \begin{equation}
   \mathcal{B}_{f}(q) := \left\{s \in \C ~\bigg|~ \sigma > 1- \frac{1}{f(q)},\  |t| < \frac{1}{\sqrt{f(q)}} \right\}, \label{ZFR}
  \end{equation}
  where $f:\Z_{\geq 2}\to \R$ satisfies $2 \leq f(q) \leq 4\log q$. Then:
  \[ \Bigg|\Re\Bigg(\frac{L'}{L}(1,\chi) - \sum_{\varrho(\chi) \cap \mathcal{B}_{f}} \frac{1}{1-\varrho}\Bigg)\Bigg| < \Big(14.5 + 2|\varrho(\chi)\cap \mathcal{B}_f|\Big) \sqrt{f(q) \log q}. \]
 \end{prop} 
 
 Theorem \ref{thm01} follows directly from Proposition \ref{AnLem} together with the classical log-free zero-density estimate (see Chapter 18 of Iwaniec--Kowalski \cite{iwankowa04}), which in particular states that
 \[ \varrho(\chi) \cap \bigg\{s\in\C ~\bigg|~ \sigma \geq 1-\frac{A}{\log q},\ |t|\leq 1\bigg\} \ll e^{cA}, \]
 where we can take $c=3$. Theorem \ref{thm01B} on the other hand will be proved using Proposition \ref{AnLem2} in subsection \ref{ssecCZFR}.

\subsection{Lemmas}\label{ssec23}
 Our starting point is the following formula:
 \begin{equation}
  \frac{L'}{L}(s,\chi) = \bigg(\sum_{\varrho(\chi)}\frac{1}{s-\varrho}\bigg) -\frac{1}{2}\log\bigg(\frac{q}{\pi}\bigg) - \frac{1}{2}\frac{\Gamma'}{\Gamma}\bigg(\frac{s+a_{\chi}}{2}\bigg), \label{ch12dav}
 \end{equation}
 where $a_{\chi} := \frac{1}{2}(1-\chi(-1))$ (see Appendix \ref{appA} for details). From the functional equation of $L(s,\chi)$, we have that, if $\varrho \in \{s\in\C ~|~ 0<\sigma < 1\}$ is a zero of $L(s,\chi)$, then $\overline{\varrho}$, $1-\varrho$ are zeros of $L(s,\overline{\chi})$, and $1-\overline{\varrho}$ is a zero of $L(s,\chi)$. Thus, by noting that $\sum_{\varrho(\chi)} (s-\varrho)^{-1} = \sum_{\varrho(\chi)} (\overline{\varrho} + (s-1))^{-1}$, we get
 \begin{equation}
  \sum_{\varrho(\chi)}\frac{\Pi_{\sigma-1}(\varrho)}{4} = \frac{1}{2}\log\bigg(\frac{q}{\pi}\bigg) + \Re\bigg(\frac{L'}{L}(\sigma,\chi)\bigg) + \frac{1}{2}\frac{\Gamma'}{\Gamma}\bigg(\frac{\sigma+a_{\chi}}{2}\bigg), \label{reLL}
 \end{equation}
 where $\Pi$ is the \emph{pairing-up function}:
 \begin{equation}
  \Pi_{\eps}(s) := \frac{1}{s + \eps} + \frac{1}{\overline{s} + \eps} + \frac{1}{1-s + \eps} + \frac{1}{1-\overline{s} + \eps}, \label{prupf}
 \end{equation}
 defined for $s,\eps \in \C$ such that $\eps \neq -s, -\overline{s}, -1+s, -1+\overline{s}$. In this notation, by setting $\sigma = 1$ in \eqref{reLL} and using the special values $\frac{\Gamma'}{\Gamma}(1) = -\gamma$ and $\frac{\Gamma'}{\Gamma}(\frac{1}{2}) = -\gamma - 2\log 2$, it follows that
 \begin{equation}
  \sum_{\varrho(\chi)} \frac{\Pi_{0}(\varrho)}{4} = \frac{1}{2}\log q + \Re\left(\frac{L'}{L}(1,\chi)\right) - \frac{1}{2}\Big(\gamma + \log 2\pi + \chi(-1)\log 2 \Big). \label{pi01}
 \end{equation}
 The next lemma estimates \eqref{pi01} for small perturbations of $\eps$ in $\Pi_{\eps}(\varrho)$.
 
 \begin{lem}\label{smllp}
  For $0< \eps < .85$, we have: \smallskip
  \begin{enumerate}[label=(\roman*)]
   \item $\displaystyle\ \sum_{\varrho(\chi)} \frac{\Pi_{\eps}(\varrho)}{4} < \frac{1}{2}\log q + \frac{1}{\eps}$; \smallskip
    
   \item $\displaystyle\ \Bigg|\sum_{\varrho(\chi)} \frac{\Pi_{\eps}(\varrho)}{4} - \frac{1}{2}\log q + \frac{1}{2}\Big(\gamma + \log 2\pi + \chi(-1)\log 2 \Big)\Bigg| < 1 + \frac{1}{\eps}$.
  \end{enumerate}
 \end{lem}
 \begin{proof}
  For $0< \sigma < 1.425$, the function $\frac{\Gamma'}{\Gamma}(\sigma)$ is strictly increasing and negative. Thus, since $\Pi_{\eps}(\varrho) > 0$, we have from \eqref{reLL} that
  \[ 0 < \sum_{\varrho(\chi)} \frac{\Pi_{\eps}(\varrho)}{4} < \Re\bigg(\frac{L'}{L}(1+\eps, \chi)\bigg) + \frac{1}{2}\log q. \]
  Moreover, since $-\frac{\zeta'}{\zeta}(\sigma) < (\sigma - 1)^{-1}$ for $1<\sigma \leq 2$,\footnote{This amounts to $(\sigma-1)\zeta(\sigma)$ being strictly increasing for $\sigma\in (1, 2]$.} it follows that
  \[ \Re\bigg(\frac{L'}{L}(1+\eps, \chi)\bigg) \leq \bigg|\dfrac{\zeta'}{\zeta}(1+\eps)\bigg| < \dfrac{1}{\eps}, \]
  which proves part (i). For the second part, we deduce from \eqref{reLL} that
  \begin{align*}
   \Bigg|\sum_{\varrho(\chi)} \frac{\Pi_{\eps}(\varrho)}{4} - \frac{1}{2}\log q + \frac{1}{2}&\Big(\gamma + \log 2\pi + \chi(-1)\log 2 \Big) \Bigg| \\
   &< \bigg|\frac{\zeta'}{\zeta}(1+\eps) \bigg| + \frac{1}{2}\Big(\gamma + \log 4\Big) < \frac{1}{\eps} + 1,
  \end{align*}
  concluding the proof.
 \end{proof}
 
 Now, let $M\geq 2$ be a fixed real number, and take the following partition of the critical strip $\{s\in \C ~|~ 0 < \sigma < 1\}$:  
 \begin{equation*}
  \begin{tikzpicture}[scale=2.9, every node/.style={scale=2.9}]
   
   \fill[pattern=crosshatch, pattern color=black!70] (0,1.55*0.666)--(0,1*0.666)--(1,1*0.666)--(1,1.55*0.666);
   \fill[pattern=crosshatch, pattern color=black!70] (0,-1.55*0.666)--(0,-1*0.666)--(1,-1*0.666)--(1,-1.55*0.666);
   
   \fill[pattern=bricks, pattern color=black!50] (.225,1*0.666)--(.775,1*0.666)--(.775,-1*0.666)--(.225,-1*0.666);
   
   \fill[pattern=north east lines, pattern color=black!50] (0,1*0.666)--(0,.376*0.666)--(.225,.376*0.666)--(.225,1*0.666);
   \fill[pattern=north east lines, pattern color=black!50] (0,-1*0.666)--(0,-.376*0.666)--(.225,-.376*0.666)--(.225,-1*0.666);
   \fill[pattern=north east lines, pattern color=black!50] (.775,1*0.666)--(.775,.376*0.666)--(1,.376*0.666)--(1,1*0.666);
   \fill[pattern=north east lines, pattern color=black!50] (.775,-1*0.666)--(.775,-.376*0.666)--(1,-.376*0.666)--(1,-1*0.666);
   
   
   \draw[-, dashed] (0,1.55*0.666)--(0,1*0.666)--(1,1*0.666)--(1,1.55*0.666);
   \draw[-, dashed] (0,-1.55*0.666)--(0,-1*0.666)--(1,-1*0.666)--(1,-1.55*0.666);
   
   \draw[-, dashed] (0,.376*0.666)--(.225,.376*0.666);
   \draw[-, dashed] (0,-.376*0.666)--(.225,-.376*0.666);
   \draw[-, dashed] (.775,.376*0.666)--(1,.376*0.666);
   \draw[-, dashed] (.775,-.376*0.666)--(1,-.376*0.666);
   
   \draw[-, dashed] (0,1*0.666)--(1,1*0.666);
   \draw[-, dashed] (0,-1*0.666)--(1,-1*0.666);
   
   \draw[-, dashed] (.225,1*0.666)--(.225,-1*0.666);
   \draw[-, dashed] (.775,1*0.666)--(.775,-1*0.666);   
   
   \draw[-{Stealth}, black] (-0.15,0)--(1.175,0);
   \draw[-{Stealth}, black] (0,-1.5*0.666)--(0,1.65*0.666);
   \draw[-] (1,-1.5*0.666)--(1,1.5*0.666);
    
   \draw (0,0)--(0,-0.05);
   \draw (0,0) node[anchor=north east, scale=.3] {$0$};
   \draw (.5,0)--(.5,-0.05); 
   \draw (1,0)--(1,-0.05);
   \draw (1,0) node[anchor=north west, scale=.3] {$1$};
   \draw (0,1*0.666)--(-0.05,1*0.666);
   \draw (0,1*0.666) node[anchor=south east, scale=.3] {$1$};
   \draw (0,-1*0.666)--(-0.05,-1*0.666);
   \draw (0,-1*0.666) node[anchor=north east, scale=.3] {$-1$};
   
   \draw (1.75,1.2*0.666) circle (.15cm);
   \draw (1.75,1.2*0.666) node[scale=0.4] {\footnotesize{$\mathcal{R}_1$}};
   \draw[->, black] (1.6,1.25*0.666) to[out=150, in=40] (0.9,1.2*0.666);
   \draw[->, black] (1.6,1.15*0.666) to[out=-110.75, in=40] (0.9,-1.2*0.666);
   
   \draw (1.75,.45*0.666) circle (.15cm);
   \draw (1.75,.45*0.666) node[scale=0.4] {\footnotesize{$\mathcal{R}_2$}};
   \draw[->, black] (1.6,.45*0.666) to[out=180, in=10] (.54,.675*0.666);
   
   \draw (1.75,-.45*0.666) circle (.15cm);
   \draw (1.75,.-.45*0.666) node[scale=0.4] {\footnotesize{$\mathcal{R}_3$}};
   \draw[->, black] (1.6,-.4*0.666) to[out=120, in=30] (0.875,.75*0.666);
   \draw[->, black] (1.6,-.5*0.666) to[out=-130, in=-25] (0.875,-.85*0.666);

   \draw (1.75,-1.2*0.666) circle (.15cm);
   \draw (1.75,-1.2*0.666) node[scale=0.4] {\footnotesize{$\widetilde{\mathcal{B}}$}};
   \draw[->, black] (1.6,-1.15*0.666) to[out=110, in=20] (0.9,-.2*0.666);
   \draw[->, black] (1.6,-1.25*0.666) to[out=130, in=0] (0.1,-.2*0.666);
   
   \draw (1.9,1.2*0.666) node[scale=0.4, anchor=west] {\footnotesize{$\displaystyle := \{s\in \C ~|~ 0 < \sigma < 1,\ |t| \geq 1 \}$}};
   
   \draw (1.9,.45*0.666) node[scale=0.4, anchor=west] {\footnotesize{$\displaystyle := \Bigg\{s\in \C ~\bigg|~ \frac{1}{M} \leq \sigma \leq 1 - \frac{1}{M},\ |t| < 1 \Bigg\}$}};
   
   \draw (1.9,-.45*0.666) node[scale=0.4, anchor=west] {\footnotesize{$\displaystyle := \left\{ s\in \C ~\Bigg|~ \begin{array}{@{}c@{}} 0 \leq \sigma(1-\sigma) \leq \tfrac{1}{M}\big(1- \tfrac{1}{M}\big), \\[.5em]  M^{-1/2} \leq |t| < 1  \end{array}\right\}$}};
   
   \draw (1.9,.-1.2*0.666) node[scale=0.4, anchor=west] {\footnotesize{$\displaystyle := \{s\in \C ~|~ 0 < \sigma < 1 \} \setminus \left(\mathcal{R}_1 \cup \mathcal{R}_2 \cup \mathcal{R}_3 \right)$}};
  \end{tikzpicture}
 \end{equation*}
 Note that, in the notation of \eqref{ZFR}, $\widetilde{\mathcal{B}} = \mathcal{B}_{1/M} \cup (1- \mathcal{B}_{1/M})$. The goal of the next lemma is to bound $\Pi_{0}(s)$ in terms of $\Pi_{\eps}(s)$ for $s$ sufficiently far away from $0$ and $1$, which in our case means ``$s$ outside of $\widetilde{\mathcal{B}}$''.
 
 \begin{lem}[Pairing-up lemma]\label{cuteineq}
  The following hold:
  \begin{enumerate}[label=(\roman*)]
   \item For $0< \Re(s) < 1$ and $\varphi := \dfrac{1+\sqrt{5}}{2}$, we have {$\displaystyle \Pi_{0}(s) > \frac{\Pi_{\varphi-1}(s)}{2\varphi-1}$}; \smallskip
   
   \item For $s\in \mathcal{R}_1\cup \mathcal{R}_2\cup \mathcal{R}_3$ and $0 < \eps < 1$, we have
   \[ |\Pi_{0}(s) - \Pi_{\eps}(s)| < 5M\eps \,\Pi_{\eps}(s).\]
  \end{enumerate}
 \end{lem}
 
 \begin{rem}
  This lemma was initially inspired by an argument attributed to U. Vorhauer used to estimate $-\sum_{\varrho(\chi)} \Re(1/\varrho)$. Although we were unable to find the original source, the argument is outlined in Exercise 8, Section 10.2 of Montgomery--Vaughan \cite{montvaug06}.
 \end{rem}
 
 \begin{proof}[Proof of Lemma \ref{cuteineq}] 
  Let $s=\sigma + it$ and $\eps> 0$. Writing $\widetilde{\sigma} := \sigma(1-\sigma)$ and $\widetilde{\sigma}_\eps := (\sigma + \eps)(1-\sigma + \eps) = \widetilde{\sigma} + \eps(1 + \eps)$, we have
  \begin{align}
   \frac{\Pi_{\eps}(s)}{2} &= \frac{\sigma + \eps}{(\sigma + \eps)^2 + t^2} + \frac{1-\sigma +\eps}{(1-\sigma +\eps)^2 + t^2} \nonumber \\
   &= \bigg(\frac{\widetilde{\sigma}_{\eps} + (1 + \widetilde{\sigma}_{\eps}) t^2 + t^4}{\widetilde{\sigma}_{\eps}^2 + ((1+2\eps)^2 - 2\widetilde{\sigma}_{\eps})t^2 + t^4}\bigg) \frac{(1+2\eps)}{1+t^2} \label{expleq0} \\
   &\hspace{-2em}= \bigg(1 + \frac{\widetilde{\sigma}(1-\widetilde{\sigma}) + 3\widetilde{\sigma} t^2   +   \eps(1+\eps) \big( 1-2\widetilde{\sigma} - \eps(1+\eps) - t^2\big)}{\widetilde{\sigma}^2 + (1 - 2\widetilde{\sigma})t^2 + t^4   +   \eps(1+\eps)\big(2\widetilde{\sigma} + \eps(1+\eps) + 2t^2\big) }\bigg) \frac{(1+2\eps)}{1+t^2}. \nonumber
  \end{align}
  Since $\eps(1+\eps)(2\widetilde{\sigma} + \eps(1+\eps) + 2t^2) > 0$, we get
  \begin{equation}
   \begin{aligned}
    &\text{\small $\displaystyle \Pi_0(s) - \frac{\Pi_{\eps}(s)}{1+2\eps} \geq$} \\
    &\hspace{2.5em}\text{\small $2\bigg(\frac{- 1 + 2\widetilde{\sigma} + \eps(1+\eps) + t^2}{\widetilde{\sigma}^2 + (1-2\widetilde{\sigma})t^2 + t^4 + \eps(1+\eps)\big(2\widetilde{\sigma} + \eps(1+\eps) + 2t^2\big)} \bigg) \frac{\eps(1+\eps)}{1+t^2}$}.\hspace{-1em}
   \end{aligned}\label{pi0piE}
  \end{equation}
  Because $\eps(1+\eps) = 1$ for $\eps = \varphi-1$, part (i) follows immediately from \eqref{pi0piE}.
  
  For part (ii), we divide the proof into three steps. Our starting point is the fact that, if $0 < \sigma < 1$, then, for $\eta > -1$,
  \begin{equation}
   \frac{\sigma}{\sigma^2 + t^2} \leq (1+\eta) \frac{\sigma + \eps}{(\sigma+\eps)^2 + t^2} \ \iff\  \eta \geq \bigg(\frac{\sigma - t^2/(\sigma + \eps)}{\sigma^2 + t^2} \bigg) \eps. \label{delub}
  \end{equation}
  
  \medskip
  \noindent
  $\bullet~\text{\underline{Step 1}:}$ {\small $\displaystyle \frac{\Pi_{\eps}(s)}{1+2\eps} < \Pi_{0}(s) < \bigg(1 + \frac{\eps^2}{1+\eps}\bigg)\Pi_{\eps}(s)$} for $s\in \mathcal{R}_1$. \smallskip
  
  For the lower bound, it suffices to note that, for every $s\in\mathcal{R}_1$, we have
  \[ -1 + 2\widetilde{\sigma} + \eps(1+\eps) + t^2 > 0, \]
  and thus, from \eqref{pi0piE}, we obtain $\Pi_0(s) - \Pi_{\eps}(s)/(1+2\eps) > 0$. For the upper bound, we use \eqref{delub}. For $s\in \mathcal{R}_1$, we have
  \[ \bigg(\frac{\sigma - t^2/(\sigma + \eps)}{\sigma^2 + t^2} \bigg)\eps < \bigg(\frac{1}{t^2} - \frac{1}{1 + \eps} \bigg)\eps \leq \bigg(1 - \frac{1}{1 + \eps} \bigg)\eps = \frac{\eps^2}{1+\eps}, \]
  and thus, taking $\eta := \eps^2/(1+\eps)$ in \eqref{delub} makes the inequality $\Pi_0(s) < (1+\eta)\Pi_{\eps}(s)$ valid for every $s \in \mathcal{R}_1$.
  
  \medskip
  \noindent
  $\bullet~\text{\underline{Step 2}:}$ {\small $\displaystyle \frac{\Pi_{\eps}(s)}{(1+2\eps)(1 + 2M \eps(1+\eps))} < \Pi_{0}(s) \leq (1 + M \eps)\, \Pi_{\eps}(s)$} for $s\in \mathcal{R}_2\cup \mathcal{R}_3$. \smallskip
  
  We start with the lower bound. The denominator of \eqref{expleq0} is always positive, and so is the numerator, for every $\eps > 0$. Thus,
  \begin{gather}
   \hspace{-12em}\Pi_0(s) - \frac{\Pi_{\eps}(s)}{(1+2\eps)(1+g(\eps))} \geq \label{comp3} \\
   \hspace{.5em}2\,\Bigg(\frac{\big(\widetilde{\sigma} - \frac{\widetilde{\sigma}_{\eps}}{1+g(\eps)}\big) + \big(\big(1-\frac{1}{1+g(\eps)}\big) + \big(\widetilde{\sigma} - \frac{\widetilde{\sigma}_{\eps}}{1+g(\eps)}\big)\big)t^2 + \big(1-\frac{1}{1+g(\eps)}\big)t^4}{\widetilde{\sigma}^2 + (1-2\widetilde{\sigma})t^2 + t^4}\Bigg)\frac{1}{1+t^2}, \nonumber
   \end{gather}
  where $g(\eps) > 0$ is some function of $\eps > 0$. Therefore, since
  \begin{align*}
   \bigg(\widetilde{\sigma} - \frac{\widetilde{\sigma}_{\eps}}{1+g(\eps)}\bigg) + &\bigg(\bigg(1-\frac{1}{1+g(\eps)}\bigg) + \bigg(\widetilde{\sigma} - \frac{\widetilde{\sigma}_{\eps}}{1+g(\eps)}\bigg)\bigg)t^2 \\
   &= \frac{g(\eps)}{1+g(\eps)} (\widetilde{\sigma} + \widetilde{\sigma}t^2 + t^2) - \frac{\eps(1+\eps)}{1+g(\eps)}(1+t^2),
  \end{align*}
  in order for \eqref{comp3} to be strictly positive it suffices to have
  \begin{equation}
   g(\eps) \geq \eps(1+\eps)\, \frac{1+t^2}{(1 + t^2)\widetilde{\sigma} + t^2} = \eps(1+\eps)\,\bigg( \widetilde{\sigma} + \frac{t^2}{1+t^2}\bigg)^{-1}. \label{gdef}
  \end{equation}
  For $s\in\mathcal{R}_2$, we have $\widetilde{\sigma} = \sigma(1-\sigma) \geq M^{-1}(1 - M^{-1}) \geq \frac{1}{2} M^{-1}$ (since $M \geq 2$), while for $s\in\mathcal{R}_3$ we have $t^2/(1+t^2) \geq \frac{1}{2}M^{-1}$. Thus, in both cases, it suffices to take $g(\eps) := 2M \eps(1+\eps)$ in order for $g$ to satisfy \eqref{gdef}, giving us the desired lower bound.  
  
  For the upper bound, we have that
  \[ \bigg(\frac{\sigma - t^2/(\sigma + \eps)}{\sigma^2 + t^2} \bigg)\eps \leq \bigg(\frac{\sigma}{\sigma^2 + t^2} \bigg)\eps \leq
   \begin{cases}
    \eps/\sigma \leq M\eps,\negphantom{\eps/\sigma \leq M\eps,}\phantom{\sigma \eps/t^2 \leq M\eps,} \ \text{ if } s\in\mathcal{R}_2, \\[.2em]
   \sigma \eps/t^2 \leq M\eps, \ \text{ if } s\in\mathcal{R}_{3};
   \end{cases} \]
  implying that taking $\eta := M\eps$ in \eqref{delub} makes $\Pi_0(s) \leq (1+\eta)\Pi_{\eps}(s)$ valid for every $s \in \mathcal{R}_2\cup\mathcal{R}_3$, concluding step 2.

  \medskip
  \noindent
  $\bullet~\text{\underline{Step 3}:}$ $|\Pi_{0}(s) - \Pi_{\eps}(s)| < 5M\eps \,\Pi_{\eps}(s)$ for $s\in \mathcal{R}_1 \cup \mathcal{R}_2 \cup \mathcal{R}_3$, $0<\eps < 1$. \smallskip
  
  Since $M\geq 2$ and $\eps < 1$, we have from step 1 that, for $s\in \mathcal{R}_1$,
  \begin{align*}
   |\Pi_0(s) - \Pi_{\eps}(s)| &< \max\bigg\{\frac{2\eps}{1+2\eps},\, \frac{\eps^2}{1+\eps}\bigg\} \,\Pi_{\eps}(s) \\
   &< \max\{2\eps, \, \eps^2 \} \,\Pi_{\eps}(s) \\
   &\leq 2\eps\, \Pi_{\eps}(s)\quad (< 5M\eps\,\Pi_{\eps}(s)),
  \end{align*}
  and from step 2 we have, for $s\in\mathcal{R}_2\cup \mathcal{R}_3$, that
  \begin{align*}
   |\Pi_0(s) - \Pi_{\eps}(s)| &\leq \max\bigg\{\frac{2\eps}{1+2\eps}\bigg(1 + \frac{M(1+\eps)}{1+2M\eps(1+\eps)}\bigg),\, M \eps \bigg\}\,\Pi_{\eps}(s) \\
   &\leq \max\bigg\{ 2M\eps\,\bigg(\frac{1}{M} + 1 + \eps\bigg),\, M\eps \bigg\}\,\Pi_{\eps}(s) \\
   &< 5M\eps \,\Pi_{\eps}(s).
  \end{align*}
  thus proving part (ii).
 \end{proof}
 
 We are now ready to prove the main propositions of this section.

\subsection{Proofs of Propositions \ref{AnLem} and \ref{AnLem2}}
 \begin{proof}[Proof of Proposition \ref{AnLem}]
  Write $\widetilde{\mathcal{S}}$ for the multiset consisting of $\varrho$, $1-\varrho$ for each $\varrho\in\mathcal{S}$, so that $|\widetilde{\mathcal{S}}| = 2|\mathcal{S}|$, and write $G = \frac{1}{2}(\gamma + \log 2\pi + \chi(-1)\log 2)$. From the definition of the pairing-up function \eqref{prupf}, we have
  \[ \frac{1}{2}\,\Re\bigg(\frac{1}{\varrho}\bigg),\ \frac{1}{2}\,\Re\bigg(\frac{1}{1-\varrho}\bigg) \leq \frac{\Pi_0(\varrho)}{4} \]
  and $\Pi_{\eps}(\varrho)/4 \leq \eps^{-1}$, for $0< \Re(\varrho) < 1$ and $\eps>0$. Thus, from \eqref{pi01} and Lemmas \ref{smllp} (ii), \ref{cuteineq} (i) we obtain:
  \begin{align*}
   &\Re\bigg(\frac{L'}{L}(1,\chi)\bigg) = \sum_{\varrho\in\widetilde{\mathcal{S}}} \frac{\Pi_{0}(\varrho)}{4} + \Bigg(\sum_{\varrho(\chi)} \frac{\Pi_{0}(\varrho)}{4} - \frac{1}{2}\log q + G - \sum_{\varrho\in\widetilde{\mathcal{S}}} \frac{\Pi_{0}(\varrho)}{4} \Bigg) \\
   &> \Re\Bigg(\sum_{\varrho \in \mathcal{S}} \frac{1}{1-\varrho}\Bigg) + \frac{1}{2\varphi-1}\Bigg(\sum_{\varrho(\chi)} \frac{\Pi_{\varphi-1}(\varrho)}{4} - \frac{1}{2}\log q + G - \sum_{\varrho\in\widetilde{\mathcal{S}}} \frac{\Pi_{\varphi-1}(\varrho)}{4}\Bigg) \\
    &\hspace{17.5em} +\, \bigg(1-\frac{1}{2\varphi-1}\bigg)\bigg(- \frac{1}{2}\log q + G\bigg) \\
   &> \Re\Bigg(\sum_{\varrho \in \mathcal{S}} \frac{1}{1-\varrho}\Bigg) -\frac{1}{2\varphi-1}\bigg(1+ \frac{1}{\varphi-1} + \frac{2|\mathcal{S}|}{\varphi-1}\bigg) - \bigg(1-\frac{1}{2\varphi -1}\bigg)\, \frac{1}{2}\log q \\
   &= \Re\Bigg(\sum_{\varrho \in \mathcal{S}} \frac{1}{1-\varrho}\Bigg) -\bigg(1-\frac{1}{\sqrt{5}}\bigg)\frac{1}{2}\log q - \frac{(\sqrt{5}+1)(2|\mathcal{S}|+3)}{2\sqrt{5}} + 1.
  \end{align*}
  The proposition then follows by rearranging the terms.
 \end{proof}
 
 \begin{proof}[Proof of Proposition \ref{AnLem2}]
  Writing $\mathcal{B} = \mathcal{B}_M$, we have from \eqref{pi01} that  
  \begin{align*}
   &\Re\Bigg(\frac{L'}{L}(1,\chi) - \sum_{\varrho(\chi) \cap \mathcal{B}} \frac{1}{1-\varrho}\Bigg) = \underbrace{\sum_{\varrho(\chi)\setminus \widetilde{\mathcal{B}}}\frac{\Pi_0(\varrho) - \Pi_{\eps}(\varrho)}{4}}_{=:\, S_1}\\
   &\hspace{11em}- \underbrace{\Bigg(\sum_{\varrho(\chi)\cap \widetilde{\mathcal{B}}} \frac{\Pi_{\eps}(\varrho)}{4} - \sum_{\varrho(\chi) \cap \mathcal{B}} \Re\bigg(\frac{1}{\varrho}\bigg)\Bigg)}_{=:\, S_2} \\
   &\hspace{7em}+ \underbrace{\Bigg(\sum_{\varrho(\chi)} \frac{\Pi_{\eps}(\varrho)}{4} - \frac{1}{2}\log q + \frac{1}{2}\Big(\gamma + \log 2\pi + \chi(-1)\log 2 \Big)\Bigg)}_{=:\, S_3}.
  \end{align*}
  From Lemma \ref{cuteineq} (ii), we have $|S_1| < 5M\eps\, \sum_{\varrho(\chi)} \Pi_{\eps}(\varrho)/4$, and thus, from Lemma \ref{smllp} (i), it follows that $|S_1| < \frac{5}{2}M\eps \log q + 5M$. Next, since $\Pi_{\eps}(\varrho)/4 < \eps^{-1}$ for $0< \Re(\varrho) < 1$, and $\Re(1/\varrho) \leq M^{-1}$ for $\varrho \in \mathcal{B}$, we have $|S_2| < 2\,|\varrho(\chi)\cap \mathcal{B}|/\eps$. Finally, from Lemma \ref{smllp} (ii), we have $|S_3| \leq 1 + \eps^{-1}$. Then, putting everything together and substituting $\eps = 1/\sqrt{M \log q}$ yields
  \begin{equation*}
   \text{\small $\displaystyle \Bigg|\Re\Bigg(\frac{L'}{L}(1,\chi) - \sum_{\varrho(\chi) \cap \mathcal{B}} \frac{1}{1-\varrho}\Bigg)\Bigg| < \Bigg(\frac{7}{2} + 2|\varrho(\chi) \cap \mathcal{B}| \Bigg)\sqrt{M \log q} + (5M + 1)$}.
  \end{equation*}
  Note that $5M+1 \leq 11M/2$. By taking $M := f(q)$ ($\leq 4\log q$), we have $f(q) \leq 2\sqrt{f(q) \log q}$, and the estimate from Proposition \ref{AnLem2} follows.
 \end{proof}
 
 \subsection{Chang's zero-free regions: Theorem \ref{thm01B}}\label{ssecCZFR} 
 For $q\geq 2$, write $q' := \prod_{p\mid q} p$, $K_q := \log q/\log q'$, and $\mathcal{P}(q)$ for the largest prime divisor of $q$. Then, there is an effectively computable constant $c>0$ such that, for every $T \geq 1$, the region
 \begin{equation*}
  \left\{s = \sigma + it ~\bigg|~ \sigma \geq 1 - c\, \min\left\{\frac{1}{\log \mathcal{P}(q)},\, \frac{\log\log q'}{(\log q') (\log 2K_q)}, \, \frac{1}{(\log qT)^{9/10}} \right\} \right\}
 \end{equation*}
 is zero-free for $|t| < T$, with the possible exception of a simple, real zero in this region in case $\chi$ is real (Theorem 10 of Chang \cite{cha14}). In particular, it follows that we can take 
 \[ f(q) := \frac{1}{c}\,\max\left\{ \log \mathcal{P}(q),\, \frac{(\log q')(\log 2K_q)}{\log\log q'},\, (\log q)^{9/10} \right\} \]
 in \eqref{ZFR}, so that the only possible element in
 \[ \varrho(\chi)\cap \bigg\{s\in\C ~\bigg|~ \sigma \geq 1-\frac{1}{f(q)},\ |t|\leq 1\bigg\} \qquad \big(\supseteq \varrho(\chi)\cap \mathcal{B}_f\big)\]
 is the potential Siegel zero, and hence the only potential term in the summation over zeros in Proposition \ref{AnLem2}. Writing $\mathcal{L}(q) := \log\log q/\log\log q'$, it becomes clear that
 \[ \frac{(\log q')(\log 2K_q)}{\log\log q'} = \left(\frac{(\log 2)\mathcal{L}(q)}{\log\log q} + \mathcal{L}(q) - 1 \right) (\log q)^{1/\mathcal{L}(q)} = o(\log q), \]
 and thus, for sufficiently smooth moduli, Chang's zero-free region is much wider than the classical one (which yields $O(\log q)$). More precisely, if $\delta>0$ and $q$ is $q^{\delta}$-smooth, then there is $N_{\delta} \in\R_{>0}$ such that $f(q)\leq c^{-1}\,\delta\log q$ for all $q>N_{\delta}$. This combined with Proposition \ref{AnLem2} proves Theorem \ref{thm01B}.{\hfill$\square$}
 
 \subsection{Proof of Corollary \ref{MC1}}\label{ssecMC1}
 Using that $\Re(\frac{1}{1-s}) = ((1-\sigma) + \frac{t^2}{1-\sigma})^{-1}$, we can apply Proposition \ref{AnLem} to deduce that a number $s=\sigma+it$ in the critical strip cannot be a zero of $L(s,\chi_D)$ if 
 \[ 1 > \bigg((1-\sigma) + \frac{t^2}{1-\sigma} \bigg)\Bigg(\underbrace{\bigg(1- \frac{1}{\sqrt{5}}\bigg) \frac{1}{2} \log|D| + \frac{L'}{L}(1,\chi_D) + O(1)}_{=:\, X}\Bigg). \]
 Multiplying by $1-\sigma$ and rearranging we get the equation
  \[ (1-\sigma)^2 X - (1-\sigma) + t^2 X < 0, \]
 which, setting $t=0$ and solving in $1-\sigma$ yields $0 < 1-\sigma < X^{-1}$. As $\big((1-\frac{1}{\sqrt{5}})\frac{1}{2}\big)^{-1} = \sqrt{5}\varphi$ for $\varphi = \frac{1+\sqrt{5}}{2}$, it follows from Theorem \ref{thm03} that, under weak uniform $abc$, we have $X^{-1} > (\sqrt{5}\varphi + o(1))\,(\log|D|)^{-1}$, thus proving the first assertion.
 
 For the second assertion, we have from Theorem \ref{thm01B} that, for any fixed $\delta >0$, there is $N_{\delta}\in\R_{>0}$ such that, if $|D| > N_{\delta}$ is $|D|^{\delta}$-smooth, then any real zero $0<\beta<1$ of $L(s,\chi_D)$ must satisfy
 \[ 0 < \frac{1}{1-\beta} < \frac{L'}{L}(1,\chi_D) + M\delta^{\frac{1}{2}}\log|D|, \]
 where $M>0$ is some absolute constant. It follows from Theorem \ref{thm03} that, under weak uniform $abc$, we must have $\beta < 1- \frac{m\delta^{-1/2} + o(1)}{\log|D|}$, where $m = M^{-1}$. Changing $o(1)$ to $o_{\delta}(1)$ allows us to drop the condition $|D| > N_{\delta}$, and since $m\delta^{-\frac{1}{2}} = o_{\delta\to 0}(\delta^{-1})$, Theorem \ref{thm01B} allows us to extend vertically the zero-free region obtained to the height $1$ box $\{s ~|~ |t|\leq 1\}$, concluding the proof. {\hfill$\square$}

\section{The bridge from \texorpdfstring{$\frac{L'}{L}(1,\chi_D)$}{L'/L(1,chi\_D)} to \texorpdfstring{$\hei(j(\tau_D))$}{ht(j(tau\_D))}}\label{sec3}
 We will now prove Theorem \ref{thm02}, which is how we connect Siegel zeros with $abc$. The proof essentially gets reduced down to a calculation once three concepts are introduced: the framework of Euler--Kronecker constants due to Ihara \cite{ihar06}, Kronecker's limit formula (KLF), and the uniform distribution of Heegner points due to Duke \cite{duk88}. After briefly describing each of these, we compute the constant $C$ from Theorem \ref{thm02} with Lemmas \ref{explCONST1}--\ref{explCONST3} and finish the proof in subsection \ref{ssec45}.
 
\subsection{Euler--Kronecker constants}
 Let $K/\Q$ be a number field, $\mathcal{O}_K$ its ring of integers, and $\zeta_K(s) = \sum_{\mathfrak{a}\subseteq\O_K} [\O_K:\mathfrak{a}]^{-s}$ its Dedekind zeta function. In 2006, Ihara \cite{ihar06} introduced and studied the \emph{Euler--Kronecker constants} $\gamma_K \in\R$, which are defined as the constant term in the Laurent expansion of $\zeta'_K/\zeta_K$ at $s=1$:
 \begin{equation*}
  \frac{\zeta'_{K}}{\zeta_K}(s) = -\frac{1}{s-1} + \gamma_K + O(s-1) \qquad (s\to 1).
 \end{equation*}
 Write $\Cl_K$ for the ideal class group of $K$, and $h_K$ ($= |\Cl_K|$) for its class number. For each $\A\in\Cl_K$, the \emph{partial zeta function} $\zeta_K(s,\A)$ is given by 
 \[ \zeta_K(s,\A) := \sum_{\substack{\mathfrak{a}\in \A \\ \mathfrak{a} \text{ integral}}} \frac{1}{[\mathcal{O}_K : \mathfrak{a}]^s}, \]
 and so we have $\zeta_K(s) = \sum_{\A \in \Cl_K} \zeta_K(s,\A)$. Roughly following Ihara's naming scheme, we define the \emph{Kronecker limits} $\kron(\A)$ as the constant term in the Laurent expansion of $\zeta'_K/\zeta_K(s,\A)$ at $s=1$. The expansion of $\zeta_K(s,\A)$ at $s=1$ is given by
 \begin{equation*}
  \zeta_K(s,\A) = \frac{\varkappa_K}{s-1} + \varkappa_K \kron(\A) + O(s-1) \qquad (s\to 1),
 \end{equation*}
 where $\varkappa_K$, which is independent of $\A$, is determined by the analytic class number formula. Hence, the Euler--Kronecker constant of $K$ is the average of its Kronecker limits:
 \begin{equation}
  \gamma_K = \frac{1}{h_K} \sum_{\A\in\Cl_K} \kron(\A). \label{eukr}
 \end{equation}
 
 Using the classical limit formulas of Kronecker (for imaginary quadratic fields) and Hecke (for real quadratic fields -- see Zagier \cite{zag75}), together with the equidistribution results in Theorem 1 of Duke \cite{duk88}, one can obtain estimates for $\gamma_{K}$ of quadratic fields in terms of cycle integrals with good error terms. We will do this in detail for the imaginary quadratic case.

\subsection{Kronecker's limit formula}
 The classical \emph{real analytic Eisenstein series} is given by
 \[ E(\tau, s) = \sum_{\substack{(m,n)\in\Z^2 \\ (m,n)\neq 0}} \frac{\Im (\tau)^s}{|m\tau + n|^{2s}}, \]
 for $\tau\in \uhp$ and $\Re(s)>1$.\footnote{Some authors consider slightly modified versions of this function, such as divided by $2$, or with the pair $(m,n)$ running through relatively prime integers (equivalent to considering $\zeta(2s)^{-1}\, E(\tau,s)$), or multiplied by $\pi^{-s} \Gamma(s)$. As we are following Siegel \cite{siegel80}, we stick to his convention.} This is the simplest example of a ``non-holomorphic modular function'', meaning it is invariant under modular transformations $\tau \mapsto \tau' = (\alpha \tau+\beta)/(\gamma \tau+\delta)$ for $\begin{psmallmatrix} \alpha & \beta \\ \gamma & \delta \end{psmallmatrix} \in \SL_{2}(\Z)$.

 In the notation of \eqref{corresp}, we have the identity
 \begin{equation}
  \zeta_{\Q(\sqrt{D})}(s,\A) = \frac{1}{w_D} \bigg(\frac{2}{\sqrt{|D|}}\bigg)^s E(\tau_{\A}, s). \label{eqPZE}
 \end{equation}
 where $w_D$ is the number of roots of unity in $\Q(\sqrt{D})$. The residue and constant term in the Laurent expansion of $E(\tau, s)$ at $s=1$ are given by
 \begin{equation}
  \begin{aligned}
   E(\tau,s) = \frac{\pi}{s-1} + \pi\,\Big(\frac{\pi}{3}\Im(\tau) \ -\, &\log\Im(\tau) \,+\, \mathcal{U}(\tau)\Big) \\
   &+\, 2\pi\big(\gamma - \log 2\big) + O(s-1), \label{LrntE}
  \end{aligned}
 \end{equation}
 where
 \begin{equation}
  \text{\small$\displaystyle \mathcal{U}(\tau) := 4\sum_{n\geq 1} \Bigg(\sum_{d\,\mid\,n} \frac{1}{d}\Bigg) \frac{\cos(2\pi n\, \Re(\tau))}{e^{2\pi n\, \Im(\tau)}} \quad \bigg(= -2\log(|\eta(\tau)|^{2}) - \frac{\pi}{3} \Im(\tau) \bigg)$}; \label{funcU}
 \end{equation}
 this expansion is also sometimes called ``Kronecker's limit formula''.\footnote{Cf. Chapter 20, \S4 of Lang \cite{lang87}.} Here, $\eta$ is \emph{Dedekind's $\eta$-function}, defined as
 \[ \eta(z) := q^{1/24} \prod_{k\geq 1} (1-q^k) \]
 for $z = x+iy \in \uhp$, where $q := e^{2\pi i z}$ and ``$q^{1/24}$'' $= e^{\pi i z/12}$. One checks that
 \begin{align}
  \mathcal{U}(z) &= 4\, \sum_{n\geq 1}\Bigg(\sum_{d\,\mid\,n} \frac{1}{d}\Bigg) \frac{\cos(2\pi n x)}{e^{2\pi n y}} = 4\, \Re\Bigg( \sum_{n\geq 1} \Bigg(\sum_{d\,\mid\,n} \frac{1}{d}\Bigg) q^n \Bigg) \nonumber \\
  &= 4\, \Re\bigg( \sum_{d\geq 1} \sum_{k\geq 1} \frac{q^{dk}}{d} \bigg) = - 4\, \Re\bigg( \sum_{k\geq 1} \log(1 - q^k) \bigg) \label{eq39sim} \\
  &= - 4\, \sum_{k\geq 1} \log|1 - q^k| = -2\log(|\eta(z)|^2) - \frac{\pi}{3} y, \nonumber
 \end{align}
 so the formula in \eqref{LrntE} is indeed equivalent to the usual formulation.
 
 \begin{xrem}
  The function $\mathcal{U}$ in \eqref{funcU} is based on Stopple's notation in p. 867 of \cite{sto06}. Although \eqref{LrntE} is usually given in terms of $\eta$, this form of the statement emphasizes the dominant part ``$\frac{\pi^2}{3}\Im(\tau_{\A})$'' in the constant term of $E(\tau_{\A},s)$.
 \end{xrem}
 
 Taking logarithmic derivatives in \eqref{eqPZE} and \eqref{LrntE} immediately yields:
 
 \begin{lem}[KLF]\label{klf}
  For each $\A \in \Cl(D)$, we have
  \begin{equation}
   \kron(\A) = \underbrace{\frac{\pi}{3}\Im(\tau_{\A}) - \log\Im(\tau_{\A}) + \mathcal{U}(\tau_{\A})}_{\A\textnormal{-dependent term}} \ \underbrace{-\,\frac{1}{2}\log|D|}_{D\textnormal{-dependent}} \ \underbrace{+\,\phantom{\frac{1}{2}}\negphantom{\frac{1}{2}} 2\gamma - \log 2}_{\textnormal{constant term}}. \label{klff}
  \end{equation}
 \end{lem}

\subsection{Duke's equidistribution theorem}
 In 1988, W. Duke showed that the set $\Lambda_D$ of Heegner points is uniformly distributed in $\SL_2(\Z)\backslash\uhp$. To state his result, consider the probability space $(\FD, \Sigma, \mu)$, where:
 \begin{itemize}
  \item $\FD\subseteq \uhp$ is the usual fundamental domain of $\SL_2(\Z)\backslash\uhp$;
  \item $\Sigma$ is the usual $\sigma$-algebra of Lebesgue measurable sets inherited from $\R^2 \supseteq \uhp$;
  \item $\mathrm{d}\mu := \frac{3}{\pi}\, \mathrm{d}x\mathrm{d}y/y^2$ for $z=x+iy \in \FD$, so that $\mu(\FD) = 1$.
 \end{itemize}
 Endowing $\uhp$ with its usual hyperbolic structure, we have the following:
 
 \begin{nthm}{Duke's Theorem}[Theorem 1 i) in \cite{duk88}]
  Let $\Omega\subseteq \FD$ be a convex set (in the hyperbolic sense) with piecewise smooth boundary. Then, there is a real number $\delta = \delta(\Omega) > 0$ such that
  \begin{equation}
   \frac{|\Lambda_{D} \cap \Omega|}{|\Lambda_{D}|} = \mu(\Omega) + O_{\Omega,\delta}(|D|^{-\delta}), \label{dkor}
  \end{equation}
  where the implied constant is ineffective.\footnote{This comes from the use of Siegel's theorem in the proof. This ineffectiveness is hence passed down to all subsequent calculations.}
 \end{nthm}

 Since hyperbolic convex subsets of $\uhp$ constitute a basis for the usual topology inherited from $\R^2 \supseteq \uhp$, the following is a direct corollary:
 
 \begin{lem}\label{dkps}
  Let $f:\FD \to \C$ be a Riemann-integrable function in $(\FD,\Sigma, \mu)$. Then:
  \[ \lim_{D\to -\infty} \Bigg(\frac{1}{h(D)} \sum_{\A\in\Cl(D)} f(\tau_{\A}) \Bigg) = \frac{3}{\pi}\, \int_{\FD} f(x + iy) \,\frac{\mathrm{d}x\mathrm{d}y}{y^2}. \]
 \end{lem}

\subsection{Computing \texorpdfstring{$C$}{C}}
 With Lemma \ref{dkps}, we are able to not only bound but also compute the average over $\A\in\Cl(D)$ of the non-dominant terms in Kronecker's limit formula \eqref{klff}. While $\kappa_1$ and $\kappa_3$ are more easily shown to be bounded without Duke's theorem, the main point of the next three lemmas is the computation of $C$ in Theorem \ref{thm02}.
 
 \begin{lem}\label{explCONST1}
  $\displaystyle \kappa_1 := -\frac{1}{6}\int_{\FD} \bigg(\log^{+}|j(z)| - 2\pi y \bigg)\,\mathrm{d}\mu = 0.011448\ldots$ 
 \end{lem}
 \begin{proof}
  Writing $q = e^{2\pi i z}$, we have $|q| = e^{-2\pi y}$, so we can rewrite $\kappa_1$ as  
  \begin{equation}
   \kappa_1 = -\frac{1}{2\pi} \int_{\FD} \log\max\big\{|j(z)|\cdot |q|,\ |q|\big\} \, \frac{\mathrm{d}x\mathrm{d}y}{y^2}. \label{esthard}
  \end{equation}
  We will prove the convergence of \eqref{esthard} in three steps, and then we will estimate its value numerically. Writing $c(n)$ for the $n$-th coefficient in the $q$-expansion of the $j$-invariant in \eqref{qexpjinv}, we have the following:
  
  \medskip
  \noindent
  $\bullet\ \underline{\text{Step 1}}\text{:}$ For every $n\geq 1$, we have $0 \leq c(n) < e^{4\pi \sqrt{n}}$.
  
  Since $(1-q^n)^{-1} = \sum_{k\geq 0} q^{nk}$, it is clear from the $q$-expansion of $j(z)$ in \eqref{qexpjinv} that the $c(n)$ are nonnegative. To show the upper bound, we use the fact that $j$ is a modular function of weight $0$ for $\SL_2(\Z)$. For every $0 < t < 1$, we have $j(i/t) = j(it)$. Thus, in terms of the $q$-expansion, rearranging this equality yields:
  \[ \sum_{n\geq 0} c(n)\bigg(\frac{e^{2\pi n /t} - e^{2\pi n t}}{e^{2\pi n (t+t^{-1})}} \bigg) = e^{2\pi/t} - e^{2\pi t}. \]
  For $n=1$ we have $c(1) = 196884 < e^{4\pi}$. For $n\geq 2$, since the $c(n)$ are nonnegative, we have
  \[ c(n) \leq \bigg(\frac{e^{2\pi/t} - e^{2\pi t}}{e^{2\pi n/t} - e^{2\pi n t}}\bigg)\, e^{2\pi n t (1+t^{-2})}, \]
  so taking $t = t(n) := 1/\sqrt{n}$, it follows that
  \begin{equation*}
   c(n) \leq \bigg(\frac{e^{2\pi\sqrt{n}} - e^{2\pi/\sqrt{n}}}{e^{2\pi n\sqrt{n}} - e^{2\pi \sqrt{n}}}\bigg)\, e^{2\pi \sqrt{n} (1+n)} = \bigg(\frac{1 - e^{-2\pi\sqrt{n}\,(1 - n^{-1})}}{1 - e^{-2\pi\sqrt{n}\,(n-1)}}\bigg)\, e^{4\pi \sqrt{n}} < e^{4\pi \sqrt{n}},
  \end{equation*}
  as desired.
  
  \medskip
  \noindent
  $\bullet\ \underline{\text{Step 2}}\text{:}$ $\big|\frac{1}{2\pi} \int_{\FD \cap \{\Im(z)\geq 16\}} \log\max\{|j(z)|\cdot |q|,\ |q|\} \, \mathrm{d}x\mathrm{d}y/y^2\big| < 10^{-21}$.
  
  For $y\geq 4$, we have $4\pi\sqrt{k} - 2\pi k y \leq -\pi k y$ for every $k\geq 1$. Thus, by step 1, for $y\geq 4$, it follows that
  \begin{equation*}
   |j(z)|\cdot |q| \leq 1 + \sum_{n\geq 0} c(n) |q|^{n+1} \leq 1 + \sum_{n\geq 0} e^{4\pi \sqrt{n+1} \,-\, 2\pi (n+1) y} \leq 1 + \sum_{n\geq 1} e^{-\pi n y}.
  \end{equation*}
  By partial summation,
  \begin{equation*}
   \sum_{n\geq 1} e^{-\pi n y} = \pi y \int_{1}^{+\infty} \lfloor t\rfloor\, e^{-\pi t y} \,\mathrm{d}t \leq \pi y \int_{1}^{+\infty} t\, e^{-\pi t y} \,\mathrm{d}t = \bigg(1+ \frac{1}{\pi y} \bigg) e^{-\pi y}.
  \end{equation*}
  Hence, as $\log(1+t) \leq t$ for all $t\geq 0$, we have
  \begin{align*}
   \bigg|\frac{1}{2\pi} \int_{\FD \cap \{\Im(z)\geq 16\}} \log\max\{|j(z)|\cdot |q|,\ |q|\} \, \frac{\mathrm{d}x\mathrm{d}y}{y^2}\bigg| &\leq \frac{1}{2\pi}\int_{16}^{+\infty} \frac{1 + (\pi y)^{-1}}{y^2\, e^{\pi y}} \,\mathrm{d}y \\
   &\leq \frac{1}{2\pi}\int_{16}^{+\infty} \frac{1 + \pi y}{y^2\, e^{\pi y}} \,\mathrm{d}y \\
   &= \frac{1}{2\pi}\, \frac{1}{16\, e^{16 \pi}}.
  \end{align*}
  Then, since $e^{4\pi} > 10^5$, we have $e^{16\pi} > 10^{20}$ and $32\pi > 10$, yielding step 2.

  \medskip
  \noindent
  $\bullet\ \underline{\text{Step 3}}\text{:}$ $\frac{1}{2\pi} \int_{\FD} \log\max\{|j(z)|\cdot |q|,\ |q|\}\, \mathrm{d}x\mathrm{d}y/y^2 = -0.011448\ldots$
  
  Since $\log\max\{|j(z)|\cdot |q|,\ |q|\}$ is continuous in the closure of $\FD \cap \{\Im(z)< 16\}$, which is compact, the integral
  \begin{equation}
   -\frac{1}{2\pi} \int_{\FD \cap \{\Im(z) < 16\}} \log\max\{|j(z)|\cdot |q|,\ |q|\} \, \frac{\mathrm{d}x\mathrm{d}y}{y^2} \label{fd16}
  \end{equation}
  converges. This, together with step 2, implies that \eqref{esthard} converges. Furthermore, by step 2, in order to obtain a computational estimate of \eqref{esthard} with 20 decimal places of accuracy, it suffices to estimate \eqref{fd16} to 20 decimal places of accuracy. Using Python's mpmath library, the first 6 decimal places are $= 0.011448\ldots$
 \end{proof}

 \begin{lem}\label{explCONST2}
  We have:
  \[\kappa_2 := \int_{\FD} \log(y)\,\mathrm{d}\mu = 1 - \log 2 + \frac{3}{\pi}\, \sum_{n\geq 1} \frac{\sin(2\pi n/3)}{n^2} \qquad (= 0.952984\ldots). \]
 \end{lem}
 \begin{proof}
  Using that $\int \cot(\vartheta)\,\mathrm{d}\vartheta = \log(\sin(\vartheta))$ and $\int \cot(\vartheta)^2\,\mathrm{d}\vartheta = - \vartheta - \cot(\vartheta)$ (omitting the integration constants), we have:
  \begin{align}
   \frac{3}{\pi}\int_{\FD} &\log(y)\frac{\mathrm{d}x\mathrm{d}y}{y^2} \nonumber \\
   &= \frac{6}{\pi}\int_{\frac{\sqrt{3}}{2}}^{+\infty} \Bigg(\int_{0}^{\frac{1}{2}} \frac{\log(y)}{y^2} \, \mathrm{d}x \Bigg) \mathrm{d}y - \frac{6}{\pi}\int_{\frac{\sqrt{3}}{2}}^{1} \Bigg(\int_{0}^{\sqrt{1-y^2}} \frac{\log(y)}{y^2} \, \mathrm{d}x \Bigg) \mathrm{d}y \nonumber \\
   &= \frac{2\sqrt{3}}{\pi} \bigg(\hspace{-.1em}\log\bigg(\frac{\sqrt{3}}{2}\bigg) + 1 \bigg) - \frac{6}{\pi} \int_{\frac{\sqrt{3}}{2}}^{1} \frac{\log(y)\sqrt{1-y^2}}{y^2} \, \mathrm{d}y \nonumber \\
   &= \frac{2\sqrt{3}}{\pi} \bigg(\hspace{-.1em}\log\bigg(\frac{\sqrt{3}}{2}\bigg) + 1 \bigg) - \frac{6}{\pi} \int_{\frac{\pi}{3}}^{\frac{\pi}{2}} \log(\sin(\vartheta))\cot(\vartheta)^2 \, \mathrm{d}\vartheta \nonumber \\
   &= 1 + \frac{3}{\pi} \int_{\frac{\pi}{3}}^{\frac{2\pi}{3}} \log(\sin(\vartheta)) \, \mathrm{d}\vartheta. \label{eqlog3pi}
  \end{align}
  Since $\sin(-i\log(\xi)) = i (1 - \xi^2)/2\xi$, making the substitution $\xi := e^{i\vartheta}$ yields:
  \begin{align*}
   \int_{\frac{\pi}{3}}^{\frac{2\pi}{3}} \log(\sin(\vartheta)) \, \mathrm{d}\vartheta &= -i \int_{\varGamma} \Bigg(\hspace{-.1em}\log\bigg(\frac{1-\xi^2}{2}\bigg) + \log\bigg(\frac{i}{\xi} \bigg) \Bigg)\, \frac{\mathrm{d}\xi}{\xi} \\
    &= i \log(2) \log(e^{\pi i/3}) + \frac{1}{2i} \Big( \Li_2(e^{2 \pi i/3}) - \Li_2(e^{4\pi i/3}) \Big) \\
    &= - \frac{\pi}{3}\,\log(2) +  \Im\big(\Li_2(\omega)\big),
  \end{align*}
  where $\omega := e^{2\pi i/3}$, $\varGamma := \{e^{i\vartheta} ~|~ \vartheta \in [\pi/3, 2\pi/3]\}$ is oriented counterclockwise, and $\Li_2(z) := \sum_{n\geq 1} z^n/n^2$ is the \emph{dilogarithm} function. Putting this together with \eqref{eqlog3pi}, we deduce the lemma by applying Lemma \ref{dkps} and by observing that for every $\vartheta \in (-\pi, \pi)$ we have $\Im(\Li_2(e^{i\vartheta})) = \sum_{n\geq 1} \sin(n\vartheta)/n^2$.
 \end{proof}

 \begin{lem}\label{explCONST3}
  Writing $\omega := e^{2\pi i/3} $, we have:
  \[ \kappa_3 := \int_{\FD} \mathcal{U}(z)\,\mathrm{d}\mu = \frac{24}{\pi}\, \Im\Bigg( \sum_{n\geq 1} \Bigg(\sum_{d\,\mid\,n} \frac{1}{d}\Bigg) \int_{\omega}^{\omega-1} e^{2\pi i n z}\frac{\mathrm{d}z}{z} \Bigg) \]
  \textnormal{($= -0.000303\ldots$)}.
 \end{lem}
 \begin{proof}
  From Equation \eqref{eq39sim}, for $z = x+iy \in \uhp$ we have
  \begin{equation}
   \mathcal{U}(z) = 4\, \Re\Bigg( \sum_{n\geq 1} \Bigg(\sum_{d\,\mid\,n} \frac{1}{d}\Bigg) e^{2\pi i n z} \Bigg) \label{uziexpl}
  \end{equation}
  Writing $\xi = \cos(\vartheta) + i \sin(\vartheta)$, we have
  \begin{align*}
   \int_{\FD} e^{2\pi i n z} \frac{\mathrm{d}x\mathrm{d}y}{y^2} &= -\frac{1}{2\pi i n}\int_{\frac{\sqrt{3}}{2}}^{1} \Big[e^{2\pi i n z} \Big|_{x= -\sqrt{1-y^2}}^{x= \sqrt{1-y^2}} \ \frac{\mathrm{d}y}{y^2} \\
   &= -\frac{1}{2\pi i n}\int_{\frac{\pi}{3}}^{\frac{2\pi}{3}} e^{2\pi i n \xi}\, \frac{\cos(\vartheta)}{\sin(\vartheta)^2} \,\mathrm{d}\vartheta \\
   &= -i\int_{\varGamma} e^{2\pi i n \xi} \,\bigg(\frac{2\xi}{\xi^2 - 1}\bigg) \, \mathrm{d}\xi,
  \end{align*}
  where $\varGamma := \{e^{i\vartheta} ~|~ \vartheta \in [\pi/3, 2\pi/3]\}$ is oriented counterclockwise. Since $2\xi/(\xi^2 - 1) = (\xi - 1)^{-1} + (\xi + 1)^{-1}$, we have
  \begin{align*}
   -i\int_{\varGamma} e^{2\pi i n \xi} \,\bigg(\frac{2\xi}{\xi^2 - 1}\bigg) \, \mathrm{d}\xi &= -i\bigg(\int_{\varGamma - 1} e^{2\pi i n \xi} \, \frac{\mathrm{d}\xi}{\xi} + \int_{\varGamma + 1} e^{2\pi i n \xi} \, \frac{\mathrm{d}\xi}{\xi} \bigg) \\
   &= 2\, \Im\Bigg(\int_{\omega}^{\omega-1} e^{2\pi i n \xi}\, \frac{\mathrm{d}\xi}{\xi} \Bigg),
  \end{align*}
  which is a real number. Thus, putting this together with \eqref{uziexpl}, we get the series in the statement of the lemma. To see that $\kappa_3$ converges, note that
  \begin{align*}
   \bigg|\int_{\omega}^{\omega-1} e^{2\pi i n \xi} \, \frac{\mathrm{d}\xi}{\xi} \bigg| \leq e^{-2\pi n \Im(\omega)}\, \frac{1}{|\omega|},
  \end{align*}
  and hence, $|\kappa_3| \leq 24(\pi |\omega|)^{-1}\, \sum_{n\geq 1} \big(\sum_{d\mid n} d^{-1}\big)\,e^{-\pi n \sqrt{3}}$, which converges.
 \end{proof}

\subsection{Proof of Theorem \ref{thm02}}\label{ssec45}
  Since $j(\tau_D)$ is integral, \eqref{heiint} gives us
  \[ \hei(j(\tau_D)) = \frac{1}{h(D)} \sum_{\A\in\Cl(D)} \log^{+}|j(\tau_{\A})|. \]
  Hence, from \eqref{eukr}, we deduce that
  \[ \gamma_{\Q(\sqrt{D})} = \frac{1}{6}\,\hei(j(\tau_D)) - \frac{1}{2}\log|D| + \big(\kappa_1 - \kappa_2 + \kappa_3 + 2\gamma - \log 2\big) + o(1) \]
  by Lemmas \ref{klf}, \ref{dkps}, and \ref{explCONST1}--\ref{explCONST3}. From the identity $\zeta_{\Q(\sqrt{D})}(s) = \zeta(s)L(s,\chi_D)$ we have $\gamma_{\Q(\sqrt{D})} = \gamma + \frac{L'}{L}(1,\chi_D)$, concluding the first part of Theorem \ref{thm02}. 
 
  To prove \eqref{hdbtt}, one simply expands the definition of $\hei(j(\tau_D))$ and uses the $q$-expansion of the $j$-invariant, thus obtaining
  \begin{equation*}
   \hei(j(\tau_D)) = \frac{1}{h(D)}\sum_{(a,b,c)} \frac{\pi\sqrt{|D|}}{a} + O(1). 
  \end{equation*}
  Rearranging the expression of the first part of Theorem \ref{thm02} yields \eqref{hdbtt}. {\null\nobreak\hfill\ensuremath{\square}}
  
 \begin{rem}
  Theorem \ref{thm02} is reminiscent of a result by Colmez \cite{col93, col98}. Writing $E_D/\C$ for an elliptic curve with complex multiplication by $\Z[\tau_D]$, Colmez showed that, writing $\hei_{\mathrm{Fal}}$ for the Faltings height,\footnote{See the diagram in subsection 0.6, p. 663 of \cite{col93} -- note that there is a typo: the upper right corner should read ``$-2\hei_{\mathrm{Fal}}(X) - \tfrac{1}{2}\log D$'' instead.}
  \[ -2\, \hei_{\mathrm{Fal}}(E_D) = \frac{1}{2}\log|D| + \frac{L'}{L}(0,\chi_D) + \log 2\pi. \]
  Together with Theorem \ref{thm02}, this yields
  \[ \hei_{\mathrm{Fal}}(E_D) = \frac{1}{12}\hei(j(\tau_D)) + \varkappa + o_{D\to-\infty}(1), \]
  where $\varkappa = \frac{C}{2} - \frac{\gamma}{2} - \log 2\pi = -2.655370\ldots$.\footnote{Cf. the standard ``$O(\log(\hei))$'' in Chapter X of Cornell--Silverman \cite{corsil86}.} It is interesting to note that in Remarque (ii), p. 364 of \cite{col98}, Colmez hinted at a possible connection between $abc$ and Siegel zeros, two years before the work of Granville and Stark.
 \end{rem}
  
\section{Estimating \texorpdfstring{$\hei(j(\tau_D))$}{ht(j(tau\_D))} with \texorpdfstring{$abc$}{abc}}\label{sec6}
 We are now going to prove Theorem \ref{thm03}. Already from Theorems \ref{thm01B} and \ref{thm02} we have
 \begin{equation}
  \liminf_{\substack{D\to-\infty \\ |D|^{o(1)}\text{-smooth}}} \frac{1}{\log|D|}\frac{L'}{L}(1,\chi_D) \geq 0, \label{eqer}
 \end{equation}
 so all we need is to work out the upper bounds. That will be done using the two uniform formulations of the $abc$-conjecture mentioned in the introduction, following Granville--Stark's method.
 
\subsection{Uniform \texorpdfstring{$abc$}{abc}-conjecture}
 We start with the usual $abc$-conjecture for number fields. Let $\rtd_{K} = |\Delta_K|^{1/[K:\Q]}$ be the root-discriminant of $K/\Q$.

 \begin{conj}[$abc$-conjecture for number fields -- see Vojta \cite{vojta87}]\label{abcplain}
  Let $K/\Q$ be a number field. For every $\eps >0$, there is a constant $\ccc{} = \ccc{K,\eps}\in\R_{\geq 0}$ such that, for any $a,b,c\in K$ with $a+b+c=0$, the following holds:
  \begin{equation}
   \lht{a:b:c} < (1+\eps)\Big(\fhr{a:b:c}{K} + \log(\rtd_K) \Big) + \ccc{K,\eps}. \label{abc0}
  \end{equation}
 \end{conj}
 
 \begin{xrem}
  For $a,b,c\in\Z$ coprime, we have $\lht{a:b:c} = \log\max\{|a|,|b|,|c|\}$ and $\fhr{a:b:c}{\Q} = \log\big(\prod_{p\mid abc} p\big)$, recovering the $abc$-conjecture for $\Z$.
 \end{xrem}

 For the method to work, we also need information about the growth of $\ccc{K,\eps}$ as a function of $K$. The approach we take is to assume that $\ccc{K,\eps}$ is dominated by the contribution ``$\log(\rtd_K)$'' of ramified primes in the inequality, in two senses of the concept of domination: either big-$O$ or small-$o$. Both have implications to the Siegel zeros problem, and we state the usual uniform $abc$-conjecture for comparison:

 \begin{conj}[Uniformity]\label{polabc}
  Let $K/\Q$ be a number field and $\eps > 0$. Then:
  \begin{enumerate}[label=(\roman*)]
   \item \textnormal{($O$-weak uniform $abc$)}~ $\ccc{K,\eps} = O_{\eps}(\log(\rtd_K))$. \smallskip
   
   \item \textnormal{(Weak uniform $abc$)}~ $\ccc{K,\eps} = o_{\eps}(\log(\rtd_{K}))$ as $\rtd_{K} \to +\infty$. \smallskip
  
   \item \textnormal{(Uniform $abc$)}~ $\ccc{K,\eps} = O_{\eps}(1)$.
  \end{enumerate}
 \end{conj}

 It is remarked in p. 510 of Granville--Stark \cite{grasta00} that Conjecture \ref{polabc} (iii) follows from Vojta's General Conjecture under the assumption that $[K:\Q]$ is bounded; consequently, so does (i) and (ii).\footnote{Writing $n = [K:\Q]$, the Hermite--Minkowski theorem says that $\rtd_K \geq \frac{\pi}{4} (n/\sqrt[n]{n!})^2$. Hence, from Stirling's formula, it follows that $\log(\rtd_K) \geq 2 + \log(\pi/4) + O(\log n/n) \gg 1$ uniformly for $n\geq 2$, and so (iii)$\implies$(ii)$\implies$(i).}

 \begin{rem}
  See also Remark 2.2.3 of Mochizuki's ``IUT IV'' \cite{IUTeichIV}, in which it is explained that the calculations of Corollary 2.2 (ii), (iii) of IUT IV can be regarded as a sort of ``weak'' version of uniform $abc$. Such version, however, is much weaker than the $O$-weak uniform $abc$ in Conjecture \ref{polabc} (i), and thus, in principle, one is not able to deduce ``no Siegel zeros'' from Corollary 2.2 of \cite{IUTeichIV} by using the methods we are employing here.
 \end{rem}

\subsection{Proof of Theorem \ref{thm03}}
 Consider the Weber modular functions $\gamma_2$, $\gamma_3$, related to the $j$-invariant via the identities
 \begin{equation}
  j(\tau) = \gamma_2(\tau)^3 = \gamma_3(\tau)^2 + 1728. \label{g23}
 \end{equation}
 For each $D<0$, the pair $(\gamma_2(\tau_D),\,\gamma_3(\tau_D))$ provides a solution $(x,y)$ to the Diophantine equation $x^3 - y^2 = 1728$. This solution lies in the extension $\hilbt{D} := \hilb{D}(\gamma_2(\tau_{D}),\gamma_3(\tau_{D}))$, where $\hilb{D} = \Q(\sqrt{D},j(\tau_D))$ is the Hilbert class field of $\Q(\sqrt{D})$. If $\gcd(D,6)=1$, then we actually have $\hilbt{D} = \hilb{D}$.\footnote{Cf. \S 72 of Weber \cite{weberIII}.} In general, we have the following:
 
 \begin{lem}[Lemma 1 of Granville--Stark \cite{grasta00}]\label{gsRD6}
 $\rtd_{\hilbt{D}} \leq 6\sqrt{|D|}$.
 \end{lem}
 
 This lemma quantifies the fact that $\hilbt{D}/\hilb{D}$ has very little ramification, which is a key point. We remark that the particular factor of ``$6$'' plays little role here, since it would be enough to have $\rtd_{\hilbt{D}} \ll_{\eps} |D|^{1/2 + \eps}$.
 
 \begin{lem}\label{uppht}
  Assume the $abc$-conjecture \eqref{abc0}. Then, for every fundamental discriminant $D<0$ and $0 < \eps < 1/5.01$, we have
  \[ \hei(j(\tau_{D})) < \frac{1}{1-5\eps}\,\Big((1+\eps)\, 3\log|D| + 6 \,\ccc{\hilbt{D},\eps}\Big) + O(1). \]
 \end{lem}
 \begin{proof}
  Writing $M := (1+\eps)\log(\rtd_{\hilbt{D}}) + \ccc{\hilbt{D},\eps}$ and applying $abc$ to the equation $\gamma_2(\tau_{D})^3 - \gamma_3(\tau_{D})^2 - 1728 = 0$, we get
  \begin{align*}
   \lht{\gamma_2^3 : \gamma_3^2 : 1728} &\stackrel{abc}{\leq} (1+\eps)\,\fhr{\gamma_2^3 : \gamma_3^2 : 1728}{K} + M \\
   &\leq (1+\eps)\,\big(\tfrac{1}{3}\,\hei(\gamma_2^3) + \tfrac{1}{2}\,\hei(\gamma_3^2)\big) + M + O(1) \\
   &\leq (1+\eps)\, \tfrac{5}{6}\,\lht{\gamma_2^3 : \gamma_3^2 : 1728} + M + O(1).
  \end{align*}
  Since $\hei(j(\tau_{D})) \leq \lht{j(\tau_{D}):j(\tau_{D})-1728:1728}$, the claim of the lemma follows by rearranging the expression above and using Lemma \ref{gsRD6}.
 \end{proof}

 \begin{proof}[Proof of Theorem \ref{thm03}]
  From Lemma \ref{gsRD6} we have $\log(\rtd_{\hilbt{D}}) = \tfrac{1}{2}\log|D| + O(1)$, so it follows from Lemma \ref{uppht} together with \eqref{eqer} that
  \begin{align*}
   \text{$O$-weak uniform $abc$ (Conjecture \ref{polabc} (i))} \ &\implies\ \limsup_{D\to-\infty} \frac{\hei(j(\tau_D))}{\log|D|} < +\infty, \\
   \text{weak uniform $abc$ (Conjecture \ref{polabc} (ii))} \ &\implies\ \limsup_{D\to-\infty} \frac{\hei(j(\tau_D))}{\log|D|} = 3;
  \end{align*}
  hence, Theorem \ref{thm03} follows directly from Theorem \ref{thm02}.
 \end{proof}

\appendix

\renewcommand{\L}{\mathcal{L}}
\newcommand{\Lc}{\varLambda_{\mathcal{L}}}
\newcommand{\selb}{\mathbf{S}}
\newcommand{\selbx}{\mathbf{S}^{\sharp}}
\renewcommand{\k}{\mathfrak{K}}
\renewcommand{\r}{\mathfrak{r}}

\section{On ``\texorpdfstring{$B + \sum 1/\varrho = 0$}{B + sum 1/rho = 0}''}\label{appA}
 Let $\chi\pmod{q}$ be a primitive Dirichlet character, $q\geq 2$. In section \ref{sec2}, we stated an expression for $\frac{L'}{L}(s,\chi)$ in \eqref{ch12dav} that is different from what usually appears in textbooks. Many investigations on Dirichlet $L$-functions have as a starting point the following classical formula:\footnote{Cf. Eqs. (17), (18), Chapter 12, p. 83 of Davenport \cite{davenport00}.}
 \begin{equation*}
  \frac{L'}{L}(s,\chi) = \bigg(\sum_{\varrho(\chi)} \frac{1}{s-\varrho}\bigg) -\frac{1}{2}\log\bigg(\frac{q}{\pi}\bigg) - \frac{1}{2}\frac{\Gamma'}{\Gamma}\bigg(\frac{s+a_{\chi}}{2}\bigg) + \bigg(B_{\chi} + \sum_{\varrho(\chi)} \frac{1}{\varrho} \bigg).
 \end{equation*}
 where $a_{\chi} := \frac{1}{2}(1-\chi(-1))$, ``$\sum_{\varrho(\chi)}$'' denotes a sum over the non-trivial zeros of $L(s,\chi)$, and $B_{\chi} \in\C$ is a constant not depending on $s$. While the fact that $\Re\big(B_{\chi} + \sum_{\varrho(\chi)} 1/\varrho\big) = 0$ is well-known, the more precise 
 \begin{equation}
  B_{\chi} + \sum_{\varrho(\chi)} \frac{1}{\varrho} = 0 \label{b0}
 \end{equation}
 does not appear to be as much so, despite holding true in similar generality. It follows from Theorem 2 (p. 257) of Ihara--Murty--Shimura \cite{ihamurshi09}, for example, that the equivalent of \eqref{b0} holds for finite-order Hecke characters in number fields (thus, in particular, for Dirichlet characters).

 While the proof in \cite{ihamurshi09} uses Weil's ``explicit formula'', it is possible to prove \eqref{b0} with a simpler, more direct analytic argument, as shown by Lucia \cite{luciaHDM}. We will reproduce Lucia's argument here in some generality, given the conspicuous absence of this fact in the literature, showing \eqref{b0} for the extended Selberg class of $L$-functions, thus highlighting that it is a relatively simple consequence of the functional equation.

 \subsection{The extended Selberg class \texorpdfstring{$\selbx$}{S\^{}\#}}
 Following Kaczorowski \cite{kacz06}, a Dirichlet series $\L(s) = \sum_{n\geq 1} a_n\,n^{-s}$ is said to be an $L$-function in the \emph{extended Selberg class} (denoted $\selbx$) if it satisfies the following properties:
 \begin{enumerate}[label=\textbf{(S\arabic*)}]
  \item\label{s1} $\sum_{n\geq 1} a_n\,n^{-s}$ is absolutely convergent for $\Re(s)>1$. \smallskip
  
  \item\label{s2} (Analytic continuation) There exists $k\in\Z_{\geq 0}$ such that $(s-1)^k\L(s)$ is entire of finite order. Write $k(\L)$ for the smallest such $k$. \smallskip
   
  \item\label{s3} (Functional equation) There exists a triple $(c, Q, \gamma_\L)$, where $c, Q$ are positive real numbers, and $\gamma_\L$ is a function of the form
  \[ \gamma_{\L}(s) := \prod_{j=1}^{f} \Gamma(\lambda_j s + \mu_j) \]
  (called a \emph{Gamma factor}), with $f\in\Z_{\geq 1}$, positive real numbers $\lambda_j$ ($1\leq j\leq f$), and complex numbers $W$, $\mu_j$ ($1\leq j \leq f$) with $|W|=1$ and $\Re(\mu_j)\geq 0$, for which the \emph{completed $L$-function}
  \begin{equation}
   \Lc(s) := c\, Q^s\, \gamma_{\L}(s)\, \L(s) \label{fnceq}
  \end{equation}
  satisfies a functional equation: $\Lc(s) = W \overline{\Lc(1-\overline{s})}$.
 \end{enumerate}
 
 This class and the more restrictive \emph{Selberg class} (with the additional Euler product and Ramanujan hypothesis properties) contain most of the $L$-functions appearing in number theory, including $L$-functions of primitive Dirichlet characters and finite-order primitive Hecke characters in number fields. From now on, let $\L$ be a fixed $L$-function in $\selbx$. See section 2 of Kaczorowski \cite{kacz06} for the facts used below.
 
 Denote by $\varrho(\L)$ the multiset of \emph{non-trivial zeros} of $\L$ (i.e., zeros of $\Lc(s)$). There is a real number $\sigma_{\L}\geq 1$ such that $1 - \sigma_{\L} \leq \Re(\varrho) \leq \sigma_{\L}$ for every non-trivial zero $\varrho$ of $\L$. Writing $N^{\pm}_{\L}(T) = \#\{\varrho(\L) ~|~ 0 \leq \pm\Im(\varrho) < T\}$, by standard techniques based on the argument principle one shows that
 \begin{equation}
  N^{-}_{\L}(T),\ N^{+}_{\L}(T) = \frac{d_{\L}}{2\pi}\,T\log T + c_{\L}T + O(\log T), \label{zeroest}
 \end{equation}
 where $d_{\L} = 2\sum_{j=1}^{f}\lambda_j$ is the \emph{degree} of $\L$, and $c_{\L}\in\R$ is some constant depending only on $\L$. This implies, in particular, that $\Lc$ has order $1$ as an entire function, and that $\sum_{\varrho(\L)\neq 0} 1/\varrho$ converges in the principal-value sense ``$\lim\limits_{T\to +\infty} \sum_{\varrho(\L)\neq 0,\, |\Im(\varrho)|\leq T}$''. Thus, from \ref{s2}, \ref{s3}, by Hadamard's factorization theorem we have
 \[ (s(1-s))^{k(\L)}\, \Lc(s) = s^{m_0}(1-s)^{m_1}\, e^{A+Bs}\,\prod_{\varrho(\L)\neq 0,1} \left(1-\frac{s}{\varrho}\right) e^{s/\varrho}, \]
 where $m_0-k(\L)$, $m_1-k(\L)$ are the order of $s=0$, $1$ as zeros of $\Lc(s)$ (hence negative if it is a pole), respectively, and $A$, $B\in\C$ are constants. Since $\Lc(0) = \overline{\Lc(1)}$, it follows that $m_0-k(\L) = m_1-k(\L)$; thus, by taking logarithmic derivatives, we derive:
 \begin{equation}
  \frac{\Lc'}{\Lc}(s) + \frac{k(\L)}{s} + \frac{k(\L)}{s-1} = \sum_{\varrho(\L)} \frac{1}{s-\varrho} + \Bigg(B + \sum_{\varrho(\L)\neq 0,1} \frac{1}{\varrho} \Bigg). \label{eqG0}
 \end{equation}
 This is our starting point.

\subsection{Proof of ``\texorpdfstring{$B + \sum 1/\varrho = 0$}{B + sum 1/rho = 0}''}
 
 \begin{thm}\label{exfor}
  For $\L \in \selbx$, in the notation of \ref{s2}, \ref{s3}, we have
  \begin{equation}
   \frac{\L'}{\L}(s) + \frac{k(\L)}{s} + \frac{k(\L)}{s-1} = \bigg(\sum_{\varrho(\L)} \frac{1}{s-\varrho}\bigg) - \log Q - \frac{\gamma'_\L}{\gamma_\L}(s). \label{exforF}
  \end{equation}
 \end{thm}
 \begin{proof}
  From \ref{s3} we have that, if $\varrho$ is a zero of $\L$, then so is $1-\overline{\varrho}$. Moreover, $\varrho(\overline{\L}) = \{1-\varrho ~|~ \varrho \in \varrho(\L)\}$, where $\overline{\L}(s) = \overline{\L(\overline{s})}$ is the \emph{dual} of $\L$. Therefore, by \eqref{fnceq} and \eqref{eqG0},
  \[ 0 = \lim_{s\to 1} \bigg(\frac{\Lc'}{\Lc}(s) + \frac{\varLambda'_{\overline{\L}}}{\varLambda_{\overline{\L}}}(1-s)\bigg) = 2\,\Re\Bigg(B + \sum_{\varrho(\L)\neq 0,1} \frac{1}{\varrho} \Bigg). \]
  Thus, in view of \eqref{fnceq}, to prove the theorem we just need to show that $\Im\big(B + \sum_{\varrho(\L)\neq 0,1} 1/\varrho\big) = 0$. Since both $\Im(\frac{\L'}{\L}(R))$, $\Im(\frac{\gamma'}{\gamma}(R)) \to 0$ as $R\to +\infty$ through the real numbers, by \eqref{fnceq} and \eqref{eqG0} it suffices to show that
  \[ \Im\bigg(\sum_{\varrho(\L)}\frac{1}{R-\varrho}\bigg) \ll \frac{\log R}{R}. \]
  Let $R>0$ be fixed and consider zeros up to height $T>0$. Writing $\varrho = \beta + i\gamma$ for a generic non-trivial zero, since $|\beta| \ll 1$, we have  
  \begin{equation*}
   \sum_{\substack{\varrho(\L) \\ |\gamma| < T}} \frac{(R-\beta) + i\gamma}{(R-\beta)^2+\gamma^2} = \sum_{\substack{\varrho(\L) \\ |\gamma| < T}} \Bigg(\frac{R-\beta+i\gamma}{R^2 + \gamma^2} + O\left(\frac{R}{(R^2 + \gamma^2)^{3/2}}\right) \Bigg). 
  \end{equation*}
  For the error term, we have  
  \begin{align*}
   \sum_{\varrho(\L)} \frac{R}{(R^2 + \gamma^2)^{3/2}} \leq \sum_{\substack{\varrho(\L) \\ |\gamma| \leq R}} \frac{1}{R^2} + \sum_{\substack{\varrho(\L) \\ |\gamma| > R}} \frac{R}{|\gamma|^3} \ll \frac{\log R}{R},
  \end{align*}
  while, for the imaginary part of the main term, since from \eqref{zeroest} it follows that $|N^{-}_{\L}(T) - N^{-}_{\L}(T)| \ll \log T$,
  \begin{align*}
   &\sum_{\substack{\varrho(\L) \\ |\gamma| < T}} \frac{\gamma}{R^2 + \gamma^2} = \int_{0}^{T} \frac{t}{R^2 + t^2}\,\mathrm{d}N^{+}_{\L}(t) - \int_{0}^{T} \frac{t}{R^2 + t^2}\,\mathrm{d}N^{-}_{\L}(t) \\[-1em]
   &\hspace{3em}= \frac{T}{R^2 + T^2}\big(N^{+}_{\L}(T) - N^{-}_{\L}(T)\big) - \int_{0}^{T}\frac{R^2+3t^2}{(R^2+t^2)^2}\big(N^{+}_{\L}(t) - N^{-}_{\L}(t)\big)\,\mathrm{d}t \\
   &\hspace{3em}\ll \frac{T\log T}{R^2 + T^2} +  \int_{0}^{T} \frac{\log t}{R^2+t^2}\,\mathrm{d}t \ll \frac{\log R}{R}. \qedhere
  \end{align*}  
 \end{proof}

\addtocontents{toc}{\protect\setcounter{tocdepth}{0}}
\section*{Acknowledgements}
 I would like to express my deep gratitude to Shinichi Mochizuki, Go Yamashita, and Andrew Granville, for their patient guidance, encouragement, and many useful critiques to this research. I would also like thank Henri Darmon, for his engaging responses to email questions, and the anonymous referee, for the invaluable recommendations on writing.
 
\addtocontents{toc}{\protect\setcounter{tocdepth}{1}}
 

\bibliographystyle{amsplain}
\bibliography{biblio_file}%
\end{document}